\documentclass{daj}
\dajAUTHORdetails{%
  title = {Around Furstenberg's Times p, Times q Conjecture: Times p-Invariant Measures with some Large Fourier Coefficients}, 
  author = {Catalin Badea and Sophie Grivaux},
  plaintextauthor = {Catalin Badea, Sophie Grivaux},
  runningtitle = {On the action of $\times _q$ on $\times_p$ invariant measures}, 
  keywords = {p-invariant measures, Furstenberg Conjecture, Fourier coefficients of continuous measures, Baire Category method},
}   

\dajEDITORdetails{%
   year={2024},
   number={10},
   received={9 June 2023},   
   published={6 November 2024},  
   doi={10.19086/da.124611},       
}   

\usepackage{amsmath,amsfonts,amsthm}
\usepackage{amssymb}
 \usepackage{blkarray}
\RequirePackage{hyperxmp}

\RequirePackage{ccicons}

 \RequirePackage{textcomp}
 
\usepackage[all]{xy}
\entrymodifiers={+!!<0pt,\fontdimen22\textfont2>}

\usepackage[abbrev,backrefs,nobysame]{amsrefs}

\DefineSimpleKey{bib}{primaryclass}{}
\DefineSimpleKey{bib}{archiveprefix}{}

\BibSpec{arXiv}{%
  +{}{\PrintAuthors}{author}
  +{,}{ \textit}{title}
  +{}{ \parenthesize}{date}
  +{,}{ arXiv }{eprint}
  +{,}{ primary class }{primaryclass}
}

\usepackage{mathrsfs}
\usepackage{mathtools}
\usepackage{tikz-cd}

\newtheorem{theorem}{Theorem}[section]
\newtheorem{lemma}[theorem]{Lemma}

\newtheorem{corollary}[theorem]{Corollary}
\newtheorem{proposition}[theorem]{Proposition}
\newtheorem{conjecture}[theorem]{Conjecture}

{\theoremstyle{definition}
}

{\theoremstyle{definition}}

\theoremstyle{definition}
\newtheorem{definition}[theorem]{Definition}
\theoremstyle{definition}
\newtheorem{question}[theorem]{Question}
\newtheorem{fact}[theorem]{Fact}

 \newtheorem*{fact*}{Fact}

{\theoremstyle{definition}\newtheorem{remark}[theorem]{Remark}}

\newcommand{\pss}[2]{\ensuremath{{\langle #1,#2\rangle}}}

\def\T{\ensuremath{\mathbb T}}
\def\R{\ensuremath{\mathbb R}}
\def\Z{\ensuremath{\mathbb Z}}

\def\C{\ensuremath{\mathbb C}}
\def\Q{\ensuremath{\mathbb Q}}
\def\N{\ensuremath{\mathbb N}}

\newcommand{\llk}{\pmb{l}_{k}}
\newcommand{\yy}{\pmb{y}}
\newcommand{\nn}{\pmb{n}}
\newcommand{\hh}{\pmb{h}}
\newcommand{\lle}{\pmb{l}}
\newcommand{\gc}{\textrm{gcd\,}}

\newcommand{\wh}[1]{\widehat{#1}}

\newcommand{\leb}{\textrm{Leb}}
\newcommand{\fr}[1]{\dfrac{1}{N_{#1}}}
\newcommand{\pp}[2]{\mathcal{P}_{#1}(\T^{#2})}
\newcommand{\muc}{\wh{\mu }}

\newcommand{\de}[1]{\delta _{\{#1\}}}
\newcommand{\gro}[1]{G_{#1}}
\newcommand{\cro}[1]{C_{#1}}
\newcommand{\dro}[1]{D_{#1}}

\newcommand{\la}{\lambda ^{p^{\nk }-1}}
\newcommand{\cnge}{(c_{n})_{n\ge 0}}
\newcommand{\cn}{(c_{n})}
\newcommand{\td}{\T^{d}}
\newcommand{\we}{w^{*}}
\newcommand{\lpj}{\lambda ^{p^{j}}}
\newcommand{\zpj}{z ^{p^{j}}}
\newcommand{\ds}{\displaystyle}
\newcommand{\pnm}{p^{N}-1 }
\newcommand{\xx}{\pmb{x}}
\newcommand{\zzk}{\pmb{z}_{k}}

\newcommand{\cc}{\pmb{C}}
\newcommand{\dd}{\pmb{D}}
\newcommand{\ff}{\pmb{F}}
\newcommand{\nk}{N_{k}}
\newcommand{\xxk}{\pmb{x}_{k}}

\newcommand{\fl}[2]{{#1 \to #2}}

\newcommand{\flw}{\xrightarrow{\text{$w^*$}}}

\newcommand{\flwco}{\xrightarrow{\text{$w^*$}}}

\makeatletter
\def\slashedarrowfill@#1#2#3#4#5{%
  $\m@th\thickmuskip0mu\medmuskip\thickmuskip\thinmuskip\thickmuskip
   \relax#5#1\mkern-7mu%
   \cleaders\hbox{$#5\mkern-2mu#2\mkern-2mu$}\hfill
   \mathclap{#3}\mathclap{#2}%
   \cleaders\hbox{$#5\mkern-2mu#2\mkern-2mu$}\hfill
   \mkern-7mu#4$%
}

\def\rightslashedarrowfilla@{%
  \slashedarrowfill@\relbar\relbar{\raisebox{1.2pt}{$\scriptscriptstyle\diagup$}}\rightarrow}
\newcommand\xslashedrightarrowa[2][]{%
  \ext@arrow 0055{\rightslashedarrowfilla@}{#1}{#2}}

\def\rightslashedarrowfillb@{%
  \slashedarrowfill@\relbar\relbar/\rightarrow}
\newcommand\xslashedrightarrowb[2][]{%
  \ext@arrow 0055{\rightslashedarrowfillb@}{#1}{#2}}

\def\rightslashedarrowfillc@{%
  \slashedarrowfill@\relbar\relbar{\raisebox{.12em}{\tiny/}}\rightarrow}
\newcommand\xslashedrightarrowc[2][]{%
  \ext@arrow 0055{\rightslashedarrowfillc@}{#1}{#2}}

\pgfdeclareshape{slash underlined}
{
  \inheritsavedanchors[from=rectangle] 
  \inheritanchorborder[from=rectangle]
  \inheritanchor[from=rectangle]{north}
  \inheritanchor[from=rectangle]{north west}
  \inheritanchor[from=rectangle]{north east}
  \inheritanchor[from=rectangle]{center}
  \inheritanchor[from=rectangle]{west}
  \inheritanchor[from=rectangle]{east}
  \inheritanchor[from=rectangle]{mid}
  \inheritanchor[from=rectangle]{mid west}
  \inheritanchor[from=rectangle]{mid east}
  \inheritanchor[from=rectangle]{base}
  \inheritanchor[from=rectangle]{base west}
  \inheritanchor[from=rectangle]{base east}
  \inheritanchor[from=rectangle]{south}
  \inheritanchor[from=rectangle]{south west}
  \inheritanchor[from=rectangle]{south east}
  \inheritanchorborder[from=rectangle]
  \foregroundpath{
    \southwest \pgf@xa=\pgf@x \pgf@ya=\pgf@y
    \northeast \pgf@xb=\pgf@x \pgf@yb=\pgf@y
    \pgf@xc=\pgf@xa
    \advance\pgf@xc by .5\pgf@xb
    \pgf@yc=\pgf@ya
    \advance\pgf@xc by -1.3pt
    \advance\pgf@yc by -1.8pt
    \pgfpathmoveto{\pgfqpoint{\pgf@xc}{\pgf@yc}}
    \advance\pgf@xc by  2.6pt
    \advance\pgf@yc by  3.6pt
    \pgfpathlineto{\pgfqpoint{\pgf@xc}{\pgf@yc}}
    \pgfpathmoveto{\pgfqpoint{\pgf@xa}{\pgf@ya}}
    \pgfpathlineto{\pgfqpoint{\pgf@xb}{\pgf@ya}}
 }
}
\tikzset{nomorepostaction/.code=\let\tikz@postactions\pgfutil@empty}

\makeatother

\newcommand{\nflw}{\xslashedrightarrowb[]{ w^*} }

\begin{document}

\begin{frontmatter}[classification=text]

\title{Around Furstenberg's Times $p$, Times $q$ Conjecture: Times $p$-Invariant Measures with some Large Fourier Coefficients\footnote{C.B. and S.G. were supported in part by
the project FRONT of the French
National Research Agency (grant ANR-17-CE40-0021) and by the Labex CEMPI (grant ANR-11-LABX-0007-01). C.B. was supported in part by the Max Planck Institute for Mathematics in Bonn.}} 

\author[C.B.]{Catalin Badea}
\author[S.G.]{Sophie Grivaux}

\begin{abstract}
For each integer \(n\ge 1\), denote by \(T_{n}\) the map \(x\mapsto nx\mod 1\) from the circle group \(\T=\R/\Z\) into itself. Let $p,q\ge 2$ be two multiplicatively independent integers. Using Baire Category arguments, we show that generically a \(T_{p}\)-invariant probability measure $\mu$ on \(\T\) with no atom has some large Fourier coefficients along the sequence $(q^n)_{n\ge 0}$. In particular, \((T_{q^{n}}\mu )_{n\ge 0}\) does not converges weak-star to the normalised Lebesgue measure on $\T$. 
This disproves a conjecture of Furstenberg
and complements previous results of Johnson and Rudolph. In the spirit of previous work by Meiri and Lindenstrauss-Meiri-Peres, we study generalisations of our main result
to certain classes  of sequences \(\cn_{n\ge 0}\) other than the sequences \((q^{n})_{n\ge 0}\), and also investigate the multidimensional setting.
\end{abstract}
\end{frontmatter}


\section{Introduction and main results}\label{Section 1}
\subsection{Synopsis}
In the late 1960s, Furstenberg proved significant results and proposed fascinating conjectures that aimed to express in various ways the heuristic principle that expansions in multiplicatively independent bases have no shared structure. For further details about this idea, readers can refer to the recent survey \cite{Shm} which also outlines some progress in Furstenberg's programme. Here, we shall list one result and three conjectures due to Furstenberg; some known partial results related to these conjectures will be mentioned in the following subsection. \emph{In all these statements,} $p,q \ge 2$ \emph{are two fixed multiplicatively independent integers}. Recall that $p,q \ge 2$ are called \emph{multiplicatively independent} if $\log p/\log q \not\in \Q$. For each integer \(n\ge 1\), denote by \(T_{n}\) the map \(x\mapsto nx\mod 1\) from the circle group \(\T=\R/\Z\), identified with \([0,1)\), into itself. A subset \(F\) of \(\T\) is said to be $T_n$-invariant if $T_n(F) \subset F$. Notice that $T_n$ shifts the $n$-ary expansion of a real number and that each map $T_n$ has many closed, infinite invariant subsets.

The following topological rigidity result has been proved in \cite{F}.
\begin{theorem}[Furstenberg]\label{thm:C0}
The only infinite closed subset \(F\) of \(\T\) which is simultaneously \(T_{p}\)- and \(T_{q}\)-invariant is \(F=\T\).
\end{theorem}

Furstenberg formulated a conjecture, called Conjecture (C1) in \cite{L2}, which is stronger than Theorem~\ref{thm:C0} and deals with the asymptotic behaviour of a  $T_p$--invariant subset under the action of $T_q$.

\begin{conjecture}\label{conj:C1}
Let $F$ be an infinite closed $T_p$-invariant subset. Then the iterates $T_{q}^{n}(F)$ converge in the Hausdorff distance to  \(\T\) as $n$ tends to infinity. 
\end{conjecture}

The measure-theoretical analogue statements of Theorem~\ref{thm:C0} and Conjecture \ref{conj:C1} are both conjectural statements. Recall that a Borel probability measure $\mu$ on  \(\T\) is said to be $T_n$-invariant if $\mu = T_n\mu$, where $T_n\mu$ is the measure defined by $T_n\mu(C) = \mu(T_n^{-1}C)$ for every measurable set $C$. There are uncountably many $T_n$-invariant, or even ergodic, measures; see for instance \cite{BenQui}*{p. 141}. The next measure-theoretical rigidity conjecture, called Conjecture (C2) in \cite{L2}, 
is the renowned $\times p$, $\times q$ conjecture of Furstenberg, one of the most fundamental open questions in ergodic theory. We say that a probability measure is \emph{continuous} if it has no atom.

\begin{conjecture}[$\times p$, $\times q$ conjecture]\label{conj:C2}
The only continuous Borel probability measure on $\T$ which is simultaneously \(T_{p}\)- and \(T_{q}\)-invariant is the (normalised) Lebesgue measure \emph{Leb}. 
\end{conjecture}

The natural analogue of Conjecture \ref{conj:C1} for measures was also conjectured by Furstenberg: this is Conjecture (C3) in \cite{L2} and concerns now the convergence in the weak-star topology of a $T_p$-invariant measure under the action of $T_q$. We shall recall the definition of \(w^{*}\) convergence in Section~\ref{Notation}.

\begin{conjecture}\label{conj:C3}
Let $\mu$ be a continuous Borel probability measure on  \(\T\) which is $T_p$-invariant. Then $T_q^n\mu$ converge \(w^{*}\) to \emph{Leb}.
\end{conjecture}

It is easy to see that  Conjecture~\ref{conj:C3} implies the $\times p$, $\times q$ conjecture (Conjecture~\ref{conj:C2}): suppose indeed that Conjecture~\ref{conj:C3}  is true, and let \(\mu \) be a continuous and simultaneously \(T_{p}\)- and \(T_{q}\)-invariant measure as in Conjecture~\ref{conj:C2}. Since \(T_q^n(\mu)={T_{q^{n}}}\mu \xrightarrow{\text{$w^*$}}\textrm{Leb}\), 
the Fourier coefficients of $\mu$ verify \(\widehat{\mu }(aq^{n})\rightarrow 0\)
for every \(a\in\Z\setminus\{0\}\). Since \(\mu \) is \(T_{q}\)-invariant, this implies that \(\wh{\mu }(a)=0\) for every \(a\in\Z\setminus\{0\}\), so that \(\mu =\textrm{Leb}\).
\par\smallskip

The main aim of this manuscript is to show that generically (in the Baire Category sense), a continuous \(T_{p}\)-invariant probability measure $\mu$ on \(\T\) has some large Fourier coefficients along the sequence $(q^n)_{n\ge 0}$. This implies that the sequence \((T_q^n(\mu))_{n\ge 0}=(T_{q^{n}}\mu )_{n\ge 0}\) does not converge \(w^{*}\) to $\textrm{Leb}$, \emph{disproving} thus Conjecture~\ref{conj:C3}. The precise statement is given in Theorem~\ref{Theorem 1} below.
\par\smallskip

It follows from our results and some results of Johnson and Rudolph in \cite{JR} that generically, in the Baire Category sense, \((T_{q^{n}}\mu )_{n\ge 0}\) does not converge \(w^{*}\) to the Lebesgue measure, but the convergence of \(T_{q^{n}}\mu\) to $\textrm{Leb}$ holds along a ``large'' sequence of integers (a sequence of upper density $1$); see Corollary~\ref{cor:jr}. This sheds some light on the complexity of the asymptotic behaviour of the action of $T_q$ on a generic $T_p$-invariant measure. In the spirit of previous work by Meiri \cite{M} and Lindenstrauss-Meiri-Peres \cite{LMP}, we study generalisations of our main result
to certain classes  of sequences \(\cn_{n\ge 0}\) other than the sequences \((q^{n})_{n\ge 0}\), and also investigate the multidimensional setting.
\par\smallskip

Our methods are mainly functional-analytic, based for instance on Baire category methods and the Hahn-Banach theorem. We also use tools from classical harmonic analysis (Fourier coefficients of measures, $p$-Bernoulli measures), ergodic theory (the periodic specification property, the ergodic decomposition theorem) and elementary number theory. It is also interesting to note that, unlike most of the works on this topic, positive entropy does not play any role in the proofs.

\subsection{Background}\label{Section1a}
Without claiming completeness, we mention some previous contributions related to Theorem~\ref{thm:C0} and Conjectures \ref{conj:C1}, \ref{conj:C2} and \ref{conj:C3}. 
\par\smallskip

Many aspects of the dynamics of subsemigroups of \((T_{n})_{n\ge 1}\) were discussed in the seminal paper \cite{F} by Furstenberg. The proof of Theorem~\ref{thm:C0} in \cite{F} used the disjointness of specific dynamical systems, a notion introduced in \cite{F}. An elementary proof of Theorem~\ref{thm:C0} has been given by Boshernitzan \cite{Bosh} and an ``effective'' version has been proved in \cite{BLMV} by Bourgain, Lindenstrauss, Michel and Venkatesh. Starting with Berend \cite{Ber}, several authors studied multidimensional generalisations of Theorem~\ref{thm:C0}.
\par\smallskip

Conjecture \ref{conj:C1} is largely open. It is known that if $F$ is a $T_p$-invariant subset of $\T$, then there exists a subsequence $(q^{n_k})$ such that $T_q^{n_k}(F)$ converges to $\T$ in the Hausdorff metric; see for instance \cite{Kra}*{Lemma 2.1}. Another result related to Conjecture \ref{conj:C1} can be found in \cite{MP}*{Th. 1.1}. Starting with the papers \cites{BerPerJLMS,AlonPeres} by Berend-Peres and Alon-Peres, several authors studied the so-called Glasner sets. A set $S$ of integers is said to be a \emph{Glasner set} if for every infinite closed subset $F$ of $\T$, there exists a sequence $(c_n)$ of elements in $S$ such that $T_{c_n}(F)$ converges to $\T$ in the Hausdorff metric. With this terminology, a result from \cite{Glasner} can be formulated as the fact that the set of integers is a Glasner set. Other quite small sets of integers are Glasner, like sets of positive (Banach) density or the sets of values assumed by any non-constant polynomial mapping the natural numbers to themselves. Note however that a finite union of lacunary sequences is not a Glasner set (\cite{BerPerJLMS}*{Th. 1.4}). Glasner sets have been also studied in the multidimensional setting.
\par\smallskip

The first result about the $\times p$, $\times q$ conjecture has been proved by Lyons in \cite{L2}, the first place where Conjecture~\ref{conj:C2} appeared in print: if \(p\) and \(q\) are relatively prime, any probability measure on \(\T\) which is \(T_{p}\)- and \(T_{q}\)-invariant and \(T_{p}\)-exact (i.e. has completely positive entropy with respect to $T_p$), must be the Lebesgue measure. 
Rudolph substantially strengthened this theorem in \cite{R}, showing that the conclusion is true with only the weaker assumption that the measure is ergodic under the joint action of \(T_{p}\) and \(T_{q}\), and of positive entropy under the action of \(T_{p}\). Johnson \cite{J} then generalised this to the case where  \(p\) and \(q\) are multiplicatively independent. A different argument, along the lines of Lyons~\cite{L2}, was given by Feldman \cite{Feldman}. Other different proofs were given by Host \cite{H} and Parry \cite{Parry}. In all these proofs the positive entropy remains a crucial assumption. The Rudolph-Johnson theorem has been used by Einsiedler and Fish \cite{EinsFish} to prove that a continuous Borel probability measure on $\T$ invariant under the action of a multiplicative semigroup with positive lower logarithmic density is the normalised Lebesgue measure. An important advance was made by Katok and Spatzier \cite{KS96}, who discovered that Rudolph’s proof can be extended to give partial information on invariant measures in much greater generality. We also mention the works \cites{Elon-padic,HochShm,Hoch,Hochman},
as well as the surveys \cites{Lind,EnsLind,LindMarg},
for an account of recent progress on measure rigidity for higher rank diagonal actions on homogeneous spaces.

\par\smallskip

Some partial results about Conjecture~\ref{conj:C3} (conjecture (C3) in \cite{L2}), which will be disproved in this manuscript, are also known. The study of convergence of the sequence \((T_{q^{n}}\mu )_{n\ge 0}\) to the Lebesgue measure for certain classes of \(T_{p}\)-invariant measures \(\mu \) lies at the core of the works of Lyons \cites{L1,L2}, Feldman and Smorodinsky \cite{FS}, Johnson and Rudolph \cite{JR}, and Host \cite{H}. Given \(p,q\ge 2\) two multiplicatively independent integers, it is shown in \cite{L1} (see also \cite{L2}) that if \(\mu \) is a non-degenerate \(p\)-Bernoulli measure, then 
\(T_{q^{n}}\mu\flwco \leb\). This is the notation for the \(w^{*}\)-convergence towards the normalised Lebesgue measure on $\T$ which is recalled in the next subsection. 
The main result of \cite{FS} states that under the same assumption, \(\mu \)-almost every \(x\in[0,1]\) is normal to the base \(q\). It is proved by Host in \cite{H} that whenever \(p\) and \(q\) are relatively prime, any measure \(\mu \in\pp{p}{}\) which is ergodic and has positive entropy with respect to \(T_{p}\) is such that \(\mu \)-almost every \(x\in[0,1]\) is normal to the base \(q\). The same statement has been proved by Lindenstrauss~\cite{Elon-padic} under the
assumption that $p$ does not divide any power of $q$ (which is weaker than Host’s
assumption). The generalisation to the case where \(p\) and \(q\) are multiplicatively independent was obtained by Hochman and Shmerkin \cite{HochShm}. Several multidimensional generalisations are discussed in \cites{Alg,Algom}. The Host-Lindenstrauss-Hochman-Shmerkin result implies easily a result by Johnson and Rudolph \cite{JR} that for every such ergodic and with positive entropy measure \(\mu \in\pp{p}{}\),
\[
\fr{}\sum_{n=0}^{N-1}T_{q^{n}}\mu \flw\leb.
\]
Johnson and Rudolph observe the following consequence: if \(\mu \in\pp{p}{}\) is ergodic and of positive entropy with respect to \(T_{p}\), then \(T_{q^{n}}\mu \flwco \leb\) on a sequence of \emph{Banach density one}
(called a sequence of uniform full density in \cite{JR}). As a consequence, they obtain that the set
\[
G'_{p,q}:=\{\mu \in\mathcal{P}_{p,c}(\T)\;;\;{T_{q^{n}}\mu \flw\textrm{Leb}}\ \textrm{along a sequence of upper density 1}\}
\]
is residual in \((\mathcal{P}_{p}(\T),w^{*})\). So, generically in the Baire Category sense, convergence of \(T_{q^{n}}\mu\) to the Lebesgue measure holds along a ``large'' sequence of integers. But the Baire Category arguments leave room for possible ``bad'' sequences where convergence to the Lebesgue measure, as predicted by Conjecture~\ref{conj:C3}, cannot be guaranteed. Quoting from \cite{JR}: ``\emph{As we have no examples showing such bad sequences can actually exist, perhaps it is possible by some more explicit investigation to eliminate these bad sequences along which convergence to the Lebesgue measure fails}''. Our first main result, which is Theorem \ref{Theorem 1} below, shows that generically such bad sequences do exist, and cannot be eliminated.

\subsection{Notation}\label{Notation}
Denote by \(\mathcal{P}(\T)\) the space of Borel  probability measures on \(\T\), and, for any $p\ge 2$, by \(\mathcal{P}_{p}(\T)\) the space of \(T_{p}\)-invariant measures \(\mu \in\mathcal{P}(\T)\). We endow \(\mathcal{P}(\T)\) with the topology of \(w^{*}\)-convergence of measures, which turns it into a compact metrizable space.
Recall that given measures $\mu_k$, $k\ge 1$, and $\mu$ belonging to \(\mathcal{P}(\T)\), we say that \(\mu_k\xrightarrow{\text{$w^*$}}\mu\) if 
$$\displaystyle\int_{\T} fd\mu_k\longrightarrow \int_{\T} fd\mu $$ as $k\longrightarrow +\infty$ for every $f\in C(\T)$, where $ C(\T)$ is the space of continuous functions on $\T$, endowed with the sup norm $||\,.\,||_{\infty,\T}$ on $\T$. This is equivalent to requiring that
$\hat{\mu}_k(a)\longrightarrow \hat{\mu}(a)$ for every $a\in\Z$, where the $a$-th Fourier coefficient of a measure $\nu\in \mathcal{P}(\T)$ is defined in this manuscript as
$$\hat{\nu}(a)=\int_{\T}z^a d\nu(z).$$
 We denote by \(\mathcal{P}_{c}(\T)\) the set of continuous (i.e. non-atomic) measures on \(\T\), and by \(\mathcal{P}_{p,c}(\T)\) the set of continuous \(T_{p}\)-invariant measures on \(\T\). 
Since \(\mathcal{P}_{p}(\T)\) is \(w^{*}\)-closed  in \(\mathcal{P}(\T)\), 
\((\mathcal{P}_{p}(\T),w^{*})\) is also a compact metrizable space. In particular, \((\mathcal{P}_{p}(\T),w^{*})\) is a Polish space, in which the Baire Category Theorem applies.  Recall that a subset of a Polish space is called \emph{residual} if it contains a dense \(G_{\delta }\) set (i.e. a countable intersection of dense open sets).
\par\smallskip

For our study of the multidimensional setting the following notation is required. For each \(d\ge 2\), we denote by \(\pp{}{d}\) the set of Borel probability measures on \(\td\), and by \(\pp{c}{d}\) the set of continuous measures \(\mu \in\pp{}{d}\). Given a matrix \(A\in M_{d}(\Z)\) with $\det(A)\neq 0$, we denote by \(T_{A}\) the associated transformation \(\xx\mapsto A\xx\mod 1\) of \(\td\) into itself. This transformation preserves the normalised Lebesgue measure on \(\td\), which we write as \(\leb_{d}\). Notice that \(T_{A}\) is an ergodic transformation of \((\td,\leb_{d})\) if and only if no eigenvalue of \(A\) is a root of unity. The set of \(T_{A}\)-invariant measures  on \(\td\) is denoted by \(\pp{A}{d}\), and \(\pp{A,c}{d}\) is the set of continuous \(T_{A}\)-invariant measures on \(\td\). 
\par\smallskip

\subsection{Main results}
Here is our first main result, showing that generically a continuous \(T_{p}\)-invariant probability measure $\mu$ on \(\T\) has some large Fourier coefficients along the sequence $(q^n)_{n\ge 0}$.
\begin{theorem}[large Fourier coefficients]\label{Theorem 1}
 Let \(p,q\ge 2\) be two 
 distinct integers. Then the set 
\[
S_{p,q} := \bigl\{\mu \in\pp{p,c}{}\;;\;\limsup\limits_{n\to+\infty}|\muc(q^n)|>0\bigr\}
\]
is residual in  \((\pp{p}{},\we)\). In particular, the set  
 \[
G_{p,q}:=\{\mu \in\mathcal{P}_{p,c}(\T)\;;\;T_{q^{n}}\mu\nflw \emph{Leb}\ \textrm{as}\ n\to+\infty\}
\]
is residual in \((\mathcal{P}_{p}(\T),w^{*})\), thus disproving Conjecture~\ref{conj:C3}.
 \end{theorem}
 
We should note that the proof of Theorem~\ref{Theorem 1} does not require that $p$ and $q$ be multiplicatively independent; if $p$ and $q$ are powers of the same integer, then a simple and direct proof of Theorem~\ref{Theorem 1} can be given. We stress that positive entropy does not play any role in the proof of Theorem~\ref{Theorem 1}.
 \par\smallskip
By combining Theorem~\ref{Theorem 1} with \cite{JR}*{Theorem~8.2}, we can derive the following corollary.
 \begin{corollary}\label{cor:jr}
 Let \(p,q\ge 2\) be two multiplicatively independent integers. Then the set of all measures $\mu \in\mathcal{P}_{p,c}(\T)$ such that
 \[
T_{q^{n}}\mu \xslashedrightarrowb[]{ w^*} \emph{Leb}\ \textrm{as}\ n\to+\infty 
\]
and 
 \[
{T_{q^{n}}\mu \xrightarrow{ w^{*}} \emph{Leb}}\ \textrm{along a sequence of upper density 1}
\]
is residual in \((\mathcal{P}_{p}(\T),w^{*})\).
 \end{corollary}

 Contrary to the proof of Theorem~\ref{Theorem 1}, the proof of Theorem~8.2. from \cite{JR} uses arguments depending on positive entropy. We shall present an alternative harmonic analysis approach to \cite{JR}*{Theorem~8.2} in Section~\ref{Section 5}.
 \par\smallskip
 
Meiri \cite{M} and Lindenstrauss, Meiri and Peres \cite{LMP} generalised the results from \cite{H} and \cite{JR} to certain classes of sequences \(\cn_{n\ge 0}\) other than the sequences \((q^{n})_{n\ge 0}\). More precisely (\cite{M}), if the sequence of remainders \((c_n\mod p^{N})_{0\le n<p^{N}}\), \(N\ge 1\), satisfies certain combinatorial properties, then every \(T_{p}\)-invariant ergodic measure \(\mu \) of positive entropy is such that \((c_n x)_{n\ge 0}\) is uniformly distributed mod \(1\) for \(\mu \)-almost every \(x\in[0,1]\).
A weaker combinatorial condition on the sequence \(\cnge\) is introduced in \cite{LMP}: if the so-called \(p\)\emph{-adic collision exponent} \(\Gamma _{p}(\cn)\) is less that \(2\), then every measure \(\mu \in\pp{p}{}\) which is ergodic and of positive entropy is \(\cn\)\emph{-generic} in the sense that
\[
\fr{}\sum_{n=0}^{N-1}T_{c_{n}}\mu \flw\leb.
\]
It follows that the set of measures \(\mu \in\pp{p,c}{}\) such that \(T_{c_{n}}\mu \flwco\leb\) along a sequence of upper density \(1\) is residual in \((\pp{p}{},w^{*})\).

Conjecture~\ref{conj:C3} thus fits in a much broader framework: given a strictly increasing sequence of integers \(\cnge\) of integers, is it true that the set 
\[
\gro{p,\cn}:=\bigl \{ \mu \in \pp{p,c}{}\;;\;T_{c_n}\mu \nflw\leb\quad \textrm{as}\ {n}\longrightarrow{+\infty}\bigr\} 
\]
is residual in \((\pp{p}{},w^{*})\)? We prove in Theorem \ref{Theorem 2.10} below a very general criterion  on the sequence \(\cnge\) implying an affirmative answer to this question. It allows to deal with most of the classes of sequences considered in \cite{M} and \cite{LMP}, and we obtain for instance the following theorem, which complements \cite{LMP}*{Th. 1.4} and \cite{M}*{Th. B}:
\begin{theorem}[linear recurrent sequences as $(c_n)$]\label{Theorem 2}
 Let \(\cnge\) be a sequence of integers satisfying a linear recursion of the form
 \[
c_{n}=a_{1}c_{n-1}+a_{2}c_{n-2}+\cdots+a_{L}c_{n-L},\quad n>L
\]
for some \(L\ge 1\) and integer coefficients \(a_{1},\dots,a_{L}\) with \(a_{L}\neq 0\). If the integers \(a_{L}\) and \(p\) are relatively prime, then the set 
\[
\gro{p,\cn}=\bigl \{ \mu \in\pp{p,c}{}\; : \;T_{c_{n}}\mu \nflw\emph{Leb}\quad \emph{as}\quad  n \to +\infty \bigr\} 
\]
is residual in \((\pp{p}{},w^{*})\).
\end{theorem}

Notice that when $L=1$ and $a_1=q$, Theorem~\ref{Theorem 2} reduces to the case studied in Theorem~\ref{Theorem 1}, but with the additional requirement that $p$ and $q$ are relatively prime.

\par\smallskip

We also mention the following related result. In \cite{BadGri}, a continuous probability measure $\mu$ on $\T$ was constructed with the property that for any increasing sequence $(c_n)$ in the multiplicative semigroup $\{p^mq^n : m,n\ge 0\}$, one has $T_{c_{n}}\mu \nflw\textrm{Leb}$. This disproved Conjectures (C4) and (C5) from \cite{L2}. However, it appears that the construction from \cite{BadGri} cannot be modified to produce a measure that is $T_p$-invariant.

\par\medskip

 We now move over to the multidimensional setting. 
 The equidistribution result \cite{H} of Host was generalised to the multidimensional setting by Meiri-Peres \cite{MP}, Host \cite{H2} himself and Algom \cite{Alg}. 
The general framework of these works is the following:
given two endomorphisms \(A\) and \(B\) of \(\T^{d}\) and a measure \(\mu \in\pp{A}{d}\), study the equidistribution properties of the sequence \((B^{n}\xx)_{n\ge 0}\) for \(\mu \)-almost every \(\xx\in\T^{d}\). This problem is studied in \cite{H2} when \(\mu \) is \(A\)-ergodic and has positive entropy, under the condition that $\det(A)$ and $\det(B)$ are relatively prime (which is exactly condition (b) of Theorem \ref{Theorem 3} below), plus some other assumptions on matrices \(A\) and \(B\). It is proved in  \cite{H2} that for every  ergodic measure \(\mu \in\pp{A}{d}\) of positive entropy, the sequence \((B^{n}\xx)_{n\ge 0}\) is uniformly distributed  in \(\td\) for \(\mu \)-almost every \(\xx\in \T^{d}\). The paper \cite{MP} considers the case where \(A\) and \(B\) are both diagonal matrices, \(A=\textrm{diag}(a_{1},\dots,a_{d})\), \(B=\textrm{diag}(b_{1},\dots,b_{d})\), with \(|a_{i}|>1\), \(|b_{i}|>1\), and \(\textrm{gcd}(a_{i},b_{i})=1\) for every \(i\in\{1\ldots d\}\).
\par\smallskip
Accordingly, we complement these results by showing a multidimensional version of Theorem \ref{Theorem 1}. 

\begin{theorem}[multidimensional setting]\label{Theorem 3}
 Let \(d\ge 2\) and let \(A,B\in M_{d}(\Z)\) with $\det (A)\neq 0$ and $\det (B)\neq 0$. Suppose that 
 \begin{enumerate}
  \item [\emph{(a)}] A is similar to a diagonal matrix \(D=\emph{diag}(\lambda _{1},\dots,\lambda _{d})\), where \(|\lambda _{j}|\neq 1\), \(1\le j\le d\);
 \end{enumerate}
 and either
 \begin{enumerate}
\item [\emph{(b)}] \(\det (A)\) and \(\det (B)\) are relatively prime;
  \end {enumerate}
  or
  \begin{enumerate}
\item [\emph{(b')}] \(A\) is upper or lower triangular.
 \end{enumerate}
 Then the set 
 \[
\gro{A,B}{}:=\bigl \{ \mu \in\pp{A,c}{d}\;;\;T_{B^{n}}\mu \nflw\emph{Leb}_{d}\bigr\} 
\]
is residual in \((\pp{A}{d},\we)\).

\end{theorem}

\subsection{Overview} 
The paper is organised as follows. We present in Section \ref{Section 2} a general criterion on a sequence \(\cnge\) of integers implying that the set $\gro{p,\cn}$
is residual in \((\pp{p}{}, \we)\). This criterion is the object of Theorem \ref{Theorem 2.10}. Its proof relies on a density result for certain classes of discrete measures in \((\pp{p}{},\we)\) (Theorem \ref{Theorem 2.1}), which is of interest in itself and involves the so-called periodic specification property of the transformation $T_p$. We present in Section \ref{Section 3} various examples of sequences considered  in \cite{M} and \cite{LMP} which satisfy the assumptions of Theorem \ref{Theorem 2.10}, and derive Theorems \ref{Theorem 1} and \ref{Theorem 2} from Theorem \ref{Theorem 2.10}. The multidimensional case is treated in Section \ref{Section 4}. Since assumption (a) of Theorem \ref{Theorem 3} does not necessarily imply that \(T_{A}:\T^{d}\longrightarrow\T^{d}\) has the periodic specification property, we need a different argument (Theorem \ref{Theorem 4.10}) in order to show the density  in \((\pp{A}{d},\we)\) of the relevant classes of \(T_{A}\)-invariant measures. We discuss in Section \ref{Section 5} a different approach to the Johnson-Rudolph result of \cite{JR} that the set \[G\,'_{\!p,q}:=\{\mu \in\mathcal{P}_{p,c}(\T)\;;\;{T_{q^{n}}\mu \flw\textrm{Leb}}\ \textrm{along a sequence of upper density 1}\}\] is residual in \((\pp{p}{},\we)\) for multiplicatively independent integers \(p\) and \(q\), and present some related results and open questions.

\section{Classes of times p-invariant measures with some large Fourier coefficients}\label{Section 2}
In the whole section, \(p\ge 2\) will be a fixed integer. Let \(\cnge\) be a strictly increasing sequence of integers. We say that \(\cnge\) satisfies \emph{assumption} (H) if the following is true:
\par\bigskip
(H)\hfill
 \begin{minipage}[c]{13cm}
  There exist finitely many nonnegative integers \(t_{1},\dots,t_{r}, h_{1},\dots,h_{d}\) with \(h_{l}\neq 0\) for every $l\in\{1,\ldots, d\}$, and an infinite subset \(I\) of \(\N\) such that for every \(N\in I\), there exist \(i\in\{1,\dots,r\}\) and \(l\in\{1,\dots,d\}\) with the property that \(h_{l}c_{n}\equiv t_{i}\mod (\pnm )\) for infinitely many integers \(n\).
 \end{minipage}
 
 \par\bigskip
Our aim in this section is to prove the following theorem:

\begin{theorem}\label{Theorem 2.10}
 Let \(\cnge\) be a strictly increasing sequence of integers satisfying assumption \emph{(H)}. Then the set 
 \[
\gro{p,\cn}=\bigl \{ \mu \in\pp{p,c}{}\;;\;T_{c_{n}}\mu \nflw\emph{\leb}\quad\textrm{as }n\longrightarrow+\infty\bigr\} 
\]
is residual in \((\pp{p}{},\we)\).
\end{theorem}
Condition (H) may look somewhat technical, but it is actually a rather weak one. We shall exhibit in Section \ref{Section 3} many examples of sequences \(\cnge\) satisfying (H). In particular, the sequence \(\cn=(q^{n})\) satisfies it for any \(q\ge 2\). This disproves Conjecture~\ref{conj:C3}.

Assumption (H) is of the same nature as the congruence assumptions mod \( p^{N}\) which appear in the works of Host \cite{H} and Meiri \cite{M}, and which are formalised in terms of \(p\)-adic collision exponent in \cite{LMP}. These two assumptions are nonetheless different, be it only because (H) involves congruences mod \((\pnm )\), while the \(p\)-adic collision exponent is defined in terms of congruence mod \( p^{N}\).
\par\medskip
Our main tool for the proof of Theorem \ref{Theorem 2.10} is a density result for certain families of discrete \(T_{p}\)-invariant measures on \(\T\).

\subsection{Density of discrete times p-invariant measures}
The periodic points of the transformation \(T_{p}\) are exactly the points \(\lambda \in\T\) such that \(\lambda ^{p^{N}}=\lambda \) for some \(N\ge 1\). In this case, the probability measure \(\mu _{\lambda }\) on \(\T \), defined  as 
\[
\mu _{\lambda }=\fr{}\sum_{j=0}^{N-1}\de{\lpj},
\]
is a discrete \(T_{p}\)-invariant measure on $\T$ whose support is the orbit of the point \(\lambda \) under the action of \(T_{p}\).
It is ergodic for \(T_{p}\), and the set of all such measures (where \(\lambda \) varies over the set of all \((p^{N}-1)\)-th roots of 1, \(N\ge 1\)) is dense in \((\mathcal{P}_{p}(\T),w^{*})\) \cites{S1,S2}. 
\par\smallskip
This density property is deeply linked to the fact that the dynamical system \((\T,T_{p})\) has the so-called \emph{specification property} introduced by Bowen in \cite{B} (see also \cites{S1,S2}). Since it will be needed in the sequel, we recall here the definition from \cite{S1}. The setting is that of compact dynamical systems \((X,T)\), where \((X,d)\) is a compact metric space and \(T\) is a continuous self-map of \(X\). This property is often referred to as the \emph{periodic specification property}, and it is the terminology we shall use here. The article \cite{KLO} contains an overview of the specification property and its many variants.

\begin{definition}\label{Definition 4}
 The system \((X,T)\) is said to have the \emph{periodic specification property} if for every \(\varepsilon >0\) there exists \(N_{\varepsilon }\in\N\) such that for every integers 
 \(0\le a_{1}\le b_{1}\) and \(0\le a_{2}\le b_{2}\) with \(a_{2}-b_{1}>N_{\varepsilon }\), for every vectors \(x_{1},x_{2}\in X\), and for every integer \(d>b_{2}-a_{1}+N_{\varepsilon }\), there exists a periodic point \(x\) for \(T\) with period $d$ such that 
 \begin{enumerate}
  \item [(i)]\(d(T^{j}x,T^{j}x_{1})<\varepsilon \)\quad  for every \(j=a_{1},\ldots, b_{1}\);
  \item [(ii)]\(d(T^{j}x,T^{j}x_{2})<\varepsilon \)\quad  for every \(j=a_{2},\ldots,b_{2}\).
 \end{enumerate}
\end{definition}
If \(x\) is periodic for \(T\) with period \(d\), the measure
\[\mu _{x}=\dfrac{1}{d}\ds\sum_{j=0}^{d-1}\de{T^{j}x}\] is called a \emph{CO-measure}. Here \emph{CO} stands for \emph{Closed-Orbit}; see for instance Sigmund~\cite{S2}. If \((X,T)\) has the specification property, the set of CO-measures is dense in the set  of \(T\)-invariant Borel probability measures on $X$ (see \cite{S2}*{Th. 1}).
\par\smallskip
Let \((\nk )_{k\ge 1}\) be a strictly increasing sequence of integers. We denote by \(\cro{p,(\nk )}\) the set of all \((p^{\nk }-1)\)-th roots of \(1\):
\[
\cro{p,(\nk )}=\bigl \{ \lambda \in\T\;;\;\lambda ^{p^{\nk }-1}=1\quad \textrm{for some}\ k\ge 1\bigr\}.
\]
Let \(\dro{p,(\nk )}\) be the family of CO-measures associated to elements \(\lambda \) of \(\cro{p,(\nk )}\):
\[
\dro{p,(\nk )}=\bigl \{\mu _{\lambda }\;;\;\lambda \in \cro{p,(\nk )} \bigr\}.
\]
\par\smallskip
We are now going to prove the following density result, which will be crucial for the proof of Theorem \ref{Theorem 1}:
\begin{theorem}\label{Theorem 2.1}
 The set \(D_{p,(N_{k})}\) is dense in \((\mathcal{P}_{p}(\T),w^{*})\).
\end{theorem}

\begin{proof}
Our aim is to show that given \(\mu \in\pp{p}{}\), \(f_{1},\dots,f_{l}\) belonging to \(C(\T)\), and \(\varepsilon >0\), there exists \(\lambda \in\cro{p,(\nk )}\) such that
\[
\Bigl | \int_{\T}f_{i}\,d\mu _{\lambda }-\int_{\T}f_{i}\,d\mu \Bigr|<\varepsilon \quad \textrm{for every}\ i\in\{1,\ldots,l\}.
\]
Since CO-measures are \(\we\)-dense in \(\pp{p}{}\), we can suppose without loss of generality that \(\mu \) is a CO-measure, which we write as
\[
\mu _{z}=\fr{}\sum_{j=0}^{N-1}\de{\zpj}\quad \textrm{for some}\ z\in\T\ \textrm{and}\ N\ge 1\ \textrm{such that}\ z^{\pnm }=1.
\]
Because Lipschitz functions, with respect to the distance induced by \(\C\) on \(\T\), are dense in $C(\T)$ by the Stone-Weierstrass theorem, we can also suppose without loss of generality that the functions \(f_{1},\dots,f_{l}\) are Lipschitz.
Let \(C>0\) be such that for every \(i\in\{1,\ldots,l\}\) and every \(z_{1},z_{2}\in\T\), \(|f_{i}(z_{1})-f_{i}(z_{2})|\le C\,|z_{1}-z_{2}|\). We are looking for \(\lambda \in\T\) and \(k\ge 1\) with \(\la=1\) such that
\[
\Bigl|\fr{}\sum_{j=0}^{N-1}f_{i}(\zpj)-\fr{k}\sum_{j=0}^{\nk -1}f_{i}(\lpj)\bigr|<\varepsilon \quad \textrm{for every}\ i\in\{1,\ldots,l\}.
\]
Fix \(\varepsilon '>0\). Let \(N_{\varepsilon '}\) be given by the specification property. Let \(k\ge 1\) be such that  \(\nk >2N_{\varepsilon }'+2\). Applying Definition \ref{Definition 4} to \(x_{1}=x_{2}=z\), \(a_{1}=0\), \(b_{1}=\nk -2N_{\varepsilon' }-2\), \(a_{2}=b_{2}=b_{1}+N_{\varepsilon '}+1\), and \(d=\nk \), we obtain the existence of \(\lambda \in\T\) with \(\lambda ^{p^{\nk }}=\lambda \) such that, for every  $j=0,\ldots, \nk -2N_{\varepsilon' }-2$,
\[
\bigl | \zpj -\lpj\bigr|<\varepsilon ', \quad \textrm{ and hence } \quad\bigl |f_{i}(\zpj)-f_{i}(\lpj) \bigr|\le C\varepsilon '.
\]
Then, for every \(i\in\{1,\ldots,l\}\),
\[
\Bigl | \dfrac{1}{\nk -2N_{\varepsilon' }-1}\sum_{j=0}^{\nk -2N_{\varepsilon' }-2}f_{i}(\zpj)-\dfrac{1}{\nk -2N_{\varepsilon' }-1}\sum_{j=0}^{\nk- 2N_{\varepsilon' }-2}f_{i}(\lpj)\Bigr|\le C\varepsilon '.
\]
Now
\begin{align*}
    \Bigl | \dfrac{1}{\nk -2N_{\varepsilon' }-1}\sum_{j=0}^{\nk -2N_{\varepsilon' }-2}f_{i}(\zpj)  
    &-\dfrac{1}{N_k}\sum_{j=0}^{\nk -1}f_{i}(\zpj)\Bigr| \le \fr{k} \sum_{j=\nk -2N_{\varepsilon'}-1}^{\nk -1}|f_{i}(\zpj)|\\ 
    &+\Bigl |\dfrac{1}{\nk -2N_{\varepsilon' }-1}-\dfrac{1}{N_k} \Bigr|\sum_{j=0}^{\nk -2N_{\varepsilon' }-2}|f_{i}(\zpj)| \\
    &\le\frac{2N_{\varepsilon' }+1}{\nk }\,\bigl|\bigl|f_{i}\bigr|\bigr|_{\infty,\T}+\Bigl (1-\dfrac{N_k-2N_{\varepsilon' }-1}{\nk } \Bigr)\,\bigl|\bigl|f_{i}\bigr|\bigr|_{\infty,\T} \\
    &\le \dfrac{4N_{\varepsilon' }+2}{\nk }\,\bigl |\bigl| f_{i} \bigr|  \bigr|_{\infty,\T}\le 6\,\dfrac{N_{\varepsilon' }}{\nk }\,\bigl |\bigl| f_{i} \bigr|  \bigr|_{\infty,\T}
\end{align*}
and
\begin{align*}
    \Bigl | \dfrac{1}{\nk -2N_{\varepsilon' }-1}\sum_{j=0}^{\nk -2N_{\varepsilon' }-2}f_{i}(\lpj)  
    &-\dfrac{1}{N_k}\sum_{j=0}^{\nk -1}f_{i}(\lpj)\Bigr| \le 6\,\dfrac{N_{\varepsilon' }}{\nk }\,\bigl |\bigl| f_{i} \bigr|  \bigr|_{\infty,\T}.
\end{align*}
Therefore
$$ \Bigl | \dfrac{1}{N_k}\sum_{j=0}^{\nk -1}f_{i}(\zpj)-\dfrac{1}{N_k}\sum_{j=0}^{\nk -1}f_{i}(\lpj)\Bigr|\le C\varepsilon '+ 12\,\dfrac{N_{\varepsilon' }}{\nk }\,\bigl |\bigl| f_{i} \bigr|  \bigr|_{\infty,\T}.
$$
Let now \(r_{k}\) be the unique integer with \(r_{k}N\le \nk <(r_{k}+1)N\). Then, proceeding in the same way as above, we obtain that
\begin{align*}
    \Bigl | \fr{k}\sum_{j=0}^{\nk -1}f_{i}(\zpj)-\dfrac{1}{r_{k}.N}\sum_{j=0}^{r_{k}.N -1}f_{i}(\zpj)\Bigr| &\le \fr{k}\sum_{j=r_{k}N}^{\nk -1}| f_{i}(\zpj)| +\Bigl |\fr{k}-\dfrac{1}{r_{k}.N} \Bigr|\sum_{j=0}^{r_{k}N-1}|f_{i}(\zpj)| \\
    &\le 2\,\dfrac{N}{\nk }\,\bigl |\bigl| f_{i} \bigr|  \bigr|_{\infty,\T}.
\end{align*} 
 Since \(N\) is fixed and \(\fl{\nk }{+\infty}\), we can choose \(\nk >2N_{\varepsilon '}+2\) sufficiently large so that 
 $$\max(12\,\frac{N_{\varepsilon' }}{\nk }, 2\,\frac{N}{\nk })\,\bigl |\bigl| f_{i} \bigr|  \bigr|_{\infty,\T}<\varepsilon '
 $$
 for every \(i\in\{1\ldots l\}\). As
 \[
\dfrac{1}{r_{k}.N}\sum_{j=0}^{r_{k}N-1}f_{i}(\zpj)=\fr{}\sum_{j=0}^{N-1}f_{i}(\zpj),
\]
 we get that
 \[
\Bigl|\fr{}\sum_{j=0}^{N-1}f_{i}(\zpj)-\fr{k}\sum_{j=0}^{\nk -1}f_{i}(\lpj)\Bigr|\le (C+2)\varepsilon '.
\]
Taking \(\varepsilon '\) so small that \((C+2)\varepsilon '<\varepsilon \) yields the result we are looking for.
\end{proof}

\begin{remark}\label{Remark 2.11}
 The argument presented above actually holds in a much more general setting, and shows the following result. Let \((X,T)\) be a dynamical system with the periodic specification property. Given a strictly increasing sequence \((\nk )_{k\ge 1}\) of integers, let
 \[
\cro{T, (\nk )}=\bigl \{ x\in X\;;\;T^{\nk }x=x\quad \textrm{for some}\ k\ge 1\bigr\} 
\]
denote the set of periodic points for \(T\) having a period within the set \(\{\nk \;;\;k\ge 1\}\). Let 
\(
\dro{T,(\nk) }=\bigl \{ \mu _{x}\;;\;x\in\cro{T,(\nk )}\bigr\} 
\). Then \(\dro{T,(\nk) }\) is dense in the set \(\mathcal{P}_{T}(X)\) of \(T\)-invariant Borel probability measures on \(X\), endowed with the \(\we\)-topology. 

Other variations of the argument are possible, involving or not the periodic specification property. 
\end{remark}

\subsection{Proof of Theorem \ref{Theorem 2.10}}
Let \(\cnge\) be a sequence of integers satisfying assumption (H), and let \(t_{1},\dots,t_{r}\), \(h_{1},\dots,h_{d}\) and \(I\subseteq\N\) be given by (H). For any \(0<\gamma <1\), consider the set 
\begin{align*}
 \gro{p,\cn}^{\,\gamma }=\bigl\{\mu \in &\pp{p}{}\;;\;\forall\,j\in\{1,\ldots, r\}\;\; \muc(t_{j})\neq 0\quad \textrm{and}\\
 &\forall\,n_{0},\ \exists\,n\ge n_{0},\ \exists\,l\in\{1,\ldots, d\}\;;\;|\muc(h_{l}c_n)|>\gamma  \min_{1\le j\le r}|\muc(t_{j})|\bigr\}.
\end{align*}
The interest of introducing this somewhat strange-looking set is the following fact.
\par\smallskip

\begin{fact}\label{Fait en plus}
If \(\mu \) belongs to \(\gro{p,\cn}^{\,\gamma }\), then there exists \(l\in\{1\ldots d\}\) such that 
\[\limsup\limits_{n\to+\infty}|\muc(h_{l}c_n)|>0.\] In particular, \(T_{c_{n}}\mu \nflw\leb\) as \({n}\longrightarrow{+\infty.}\)
\end{fact}

\begin{proof}
Let \(\mu \in\gro{p,\cn}^{\,\gamma }\). There exists \(l\in\{1\ldots d\}\) such that $|\muc(h_{l}c_n)|>\gamma  \min_{1\le j\le r}|\muc(t_{j})|$ for infinitely many $n$'s, and so
\[\limsup\limits_{n\to+\infty}|\muc(h_{l}c_n)|\ge\gamma  \min_{1\le j\le r}|\muc(t_{j})|>0.\]
Hence $\widehat{T_{c_n}\mu}(h_l) {\xymatrix@C20pt{\ar[r]\save[]+<0.45cm,-0.1pt>*{
\scriptstyle/}\restore&}}0$ as \({n}\longrightarrow{+\infty}\), and as $h_l\neq 0$ this implies
that \(T_{c_{n}}\mu \nflw\leb\).
\end{proof}

\par\smallskip
We first prove:

\begin{lemma}\label{Lemma 2.3}
 For every \(0<\gamma <1\), the set \(\gro{p,\cn}^{\,\gamma }\) is a dense \(G_{\delta }\) subset of \((\pp{p}{}, \we)\).
\end{lemma}

\begin{proof}
 The set \(\gro{p,\cn}^{\,\gamma }\) is clearly \(G_{\delta }\) in \((\pp{p}{}, \we)\), so we only need to show that it is dense. 
 Order the infinite set \(I\subseteq\N\) as a strictly increasing sequence \((\nk )_{k\ge 1}\). 
 
 We have the following:
 
 \begin{fact}\label{Fact 3.3}
  Let \(\mu \) belong to \(\dro{p,(\nk )}\). There exist \(i\in\{1,\ldots,r\}\) and $l\in\{1,\ldots,d\}$ such that 
  \(\muc{(h_lc_n)}=\muc(t_{i})\) for infinitely many integers \(n\).
 \end{fact}
 
\begin{proof}[Proof of Fact \ref{Fact 3.3}]
There exists \(\lambda \in \cro{p,(\nk )}\), with \(\la=1\) for some \(k\ge 1\), such that 
\[
\mu =\mu _{\lambda }=\fr{k}\sum_{j=0}^{\nk -1}\de{\lambda ^{p^{j}}}.
\]
For every \(a\in\Z\), we have
\[
\muc(a)=\fr{k}\sum_{j=0}^{\nk -1}\lambda ^{a.p^{j}}.
\]
Since \(\nk \in I\), there exist \(i\in\{1,\dots,r\}\) and \(l\in\{1,\dots,d\}\) such that \(h_{l}c_{n}\equiv t_{i}\mod (p^{\nk} -1)\) for infinitely many \(n\)'s. Hence 
\[h_{l}c_{n}p^{j}\equiv t_{i}p^{j}\mod (p^{\nk} -1)\quad \textrm{for every } 0\le j<\nk . \] As 
\(\la=1\), it follows that \(\lambda ^{h_{l}c_{n}p^{j}}=\lambda ^{t_{i}p^{j}}\). This yields that \(\muc(h_{l}c_{n})=\muc(t_{i})\). 
\end{proof}

Let now \(\mathcal{V}\) be a non-empty open subset of \((\pp{p}{},\we)\). By Theorem \ref{Theorem 2.1}, there exists \(\mu \in\dro{p,(\nk )}\cap\mathcal{V}\). Let \(i\in\{1,\dots,r\}\) and \(l\in\{1,\dots,d\}\) be such that \(\muc(h_{l}c_{n})=\muc(t_{i})\) for every \(n\) belonging to a certain infinite subset \(D\) of \(\N\). Then 
\[
 \limsup_{n\to+\infty}|\muc(h_{l}c_{n})|\ge |\muc(t_{i})|\ge \min_{1\le j\le r}|\muc(t_{j})|,
\]
and if \(\muc(t_{j})\neq 0\) for every \(j\in\{1,\ldots,r\}\), then \(\mu \) belongs to \(\gro{p,\cn}^{\gamma }\). Hence \(\gro{p,\cn}^{\gamma }\cap\mathcal{V}\neq\emptyset\) in this case.
\par\smallskip
Suppose now that \(\min\limits_{1\le j\le r}|\muc(t_{j})|=0\), and write \(\{1,\dots,r\}=I\,\cup\,J\), where
\[
I=\bigl\{j\in\{1,\ldots, r\}\;;\;\muc(t_{j})= 0\bigr\}\quad \textrm{and}\quad J=\bigl\{j\in\{1,\ldots, r\}\;;\;\muc(t_{j})\neq 0\bigr\}.
\]
For any \(0<\rho <1\), consider the measure \(\mu _{\rho }=\rho \delta _{1}+(1-\rho )\mu \): it is \(T_{p}\)-invariant and belongs to \(\mathcal{V}\) if \(\rho \) is sufficiently small. For every \(j\in\{1,\ldots, r\}\), \(\muc_{\rho }(t_{j})=\rho +(1-\rho )\muc(t_{j})\), so that \(\muc_{\rho }(t_{j})=\rho >0\) for every \(j\in I\). If \(\rho \) is sufficiently small, \(|\muc_{\rho }(t_{j})|>0\) for every \(j\in J\) and thus \(\min\limits_{1\le j\le r}|\muc_{\rho }(t_{j})|>0\). Since
$
\muc_{\rho }(h_{l}c_{n})=\rho +(1-\rho )\muc(t_{i})=\muc_{\rho }(t_{i})$ for every $n\in D,
$
it follows that \(\muc_{\rho }\) belongs to \(\gro{p,\cn}^{\gamma }\cap\mathcal{V}\) in this case as well. 
Lemma \ref{Lemma 2.3} is proved.
\end{proof}

The last step in our proof of Theorem \ref{Theorem 2.10} is the following classical fact:
\begin{fact}\label{Fact 3.4}
 The set \(\pp{p,c}{}\) is a dense \(G_{\delta }\) subset of \((\pp{p}{},\we)\).
\end{fact}
\begin{proof}
 It is a known result that if \(X\) is a Polish space, the set \(\mathcal{P}_{c}(X)\) of continuous probability measures on \(X\) is a \(G_{\delta }\) subset of the set \(\mathcal{P}(X)\) of all Borel probability measures on \(X\), endowed with the \(\we\)-topology (see for instance \cite{DGS}*{Proposition 2.16} or \cite{GM}*{Fact 3.2}). So \(\pp{p,c}{}\) is \(G_{\delta }\) in \((\pp{p}{},\we)\). The density of \(\pp{p,c}{}\) in \(\pp{p}{}\) is proved in \cite{S1}, using the density of CO-measures in \(\pp{p}{}\). It can also be retrieved by using the following elementary observation: there exists a sequence \((\mu _{k})_{k\ge 1}\) of elements of \(\pp{p,c}{}\) such that \(\mu _{k}\flwco \delta_1\).
 \par\smallskip
 The measures \(\mu _{k}\) can be constructed  as Cantor-type measures, also called \(p\)-Bernoulli in \cite{L1}, \cite{L2}, or \cite{FS}: given a \(p\)-tuple
 \(\Theta =(\theta _{0},\theta _{1},\dots,\theta _{p-1})\) of elements of \((0,1)\) with \(\sum_{j=0}^{p-1}\theta _{i}=1\), let \(m_{\Theta }\) be the product measure
 \[
m_{\Theta }=\bigotimes_{n\ge 1}\,\,\Bigl (\, \sum_{j=0}^{p-1}\theta _{j}\de{j}\,\Bigr)\quad \textrm{on}\quad \{0,1,\dots,p-1\}^{\N}, 
\]
and let \(\mu _{\Theta }\in\pp{}{}\) be the image measure of \(m_{\Theta }\) by the map \(\fl{\Phi :\{0,1,\dots,p-1\}^{\N}}{\T}\) defined by 
\[
\Phi \bigl ( (\omega _{n})_{n\ge 1}\bigr)=\exp\Bigl (2i\pi \sum_{n\ge 1}\omega _{n}p^{-n} \Bigr).
\]
Each measure \(\mu _{\Theta }\) is easily seen to belong to \(\pp{p,c}{}\), and 
\({\mu _{\Theta }}\flwco{\delta _{1}}\) as \(\fl{\Theta }{(1,0,\dots,0).}\) 
\par\smallskip
Once we have obtained a sequence \((\mu _{k})_{k\ge 1}\) of elements of \(\mathcal{P}_{p,c}(\T)\) such that \(\mu _{k}\xrightarrow{\text{$w^{*}$}}\delta _{1}\), the density of \(\mathcal{P}_{p,c}(\T)\)in \(\mathcal{P}_{p}(\T)\) immediately follows, since for each measure \(\mu \in\mathcal{P}_{p}(\T)\), \((\mu _{k}*\mu )_{k\ge 1}\) is a sequence of \(T_{p}\)-invariant continuous measures which converges \(w^{*}\) to \(\mu\).
\end{proof}

\begin{proof}[Proof of Theorem \ref{Theorem 2.10}]
 It follows from Lemma \ref{Lemma 2.3} and Fact \ref{Fact 3.4} and from the Baire Category Theorem  that for any $\gamma\in(0,1)$, the set \(\gro{p,\cn}^{\gamma}\cap\pp{p,c}{}\) is residual in \((\pp{p}{},\we)\). By Fact \ref{Fait en plus}, the set 
\[
\bigl\{\mu \in\pp{p,c}{}\;;\;\exists l\in\{1,\ldots, d\}\;\;\limsup\limits_{n\to+\infty}|\muc(h_{l}c_n)|>0\bigr\}
\]
is residual in \((\pp{p}{},\we)\), and hence
\(\gro{p,\cn}{}\) is residual as well. Theorem \ref{Theorem 2.10} is proved.
\end{proof}

\section{Proofs of Theorems \ref{Theorem 1} and \ref{Theorem 2}, and further examples}\label{Section 3}
In this section, we apply Theorem \ref{Theorem 2.10} to various classes of sequences \(\cnge\), and show that generically in the Baire Category sense, a measure \(\mu \in\pp{p}{}\) has infinitely many ``large'' Fourier coefficients along the sequence \(\cnge\), or along some dilated sequence \((a.c_{n})_{n\ge 0}\) for some \(a\in\Z\setminus\{0\}\). We begin by proving Theorem \ref{Theorem 1}.

\subsection{Disproving Conjecture (C3): proof of Theorem \ref{Theorem 1}}
Let \(p\ge 2\). Given another integer \(q\ge 2\) (not necessarily multiplicatively independent from \(p\)), we consider the sequence \(c_{n}=q^{n}\), \(n\ge 0\). In order to show that the sets
\[
S_{p,q} := \bigl\{\mu \in\pp{p,c}{}\;;\;\limsup\limits_{n\to+\infty}|\muc(q^n)|>0\bigr\}
\]
and 
\[\gro{p,q}{}=\{\mu \in\pp{p,c}{}\;;\;T_{q^{n}}\mu \nflw\leb\}\] are dense in \((\mathcal{P}_{p,c}(\T),\we)\)  it suffices to show that this sequence \(\cnge\) satisfies assumption (H), and then to apply Theorem \ref{Theorem 2.10}. That assumption (H) is satisfied is a consequence of the following lemma, which relies on considerations from elementary number theory:

\begin{lemma}\label{Lemma 3.1}
 Let \(p,q\ge 2\). There exists an integer $N_0\ge 1$ such that for every \(a\ge 1\) sufficiently large, the following assertion holds:
 
 \par\smallskip
for every \(N\in I:= N_0.\N+1\),   there exists an integer \(r_{N}\ge 1\) such that
\[
q^{a+k.r_{N}}\equiv q^{a}\mod (\pnm )\quad \textrm{for every}\ k\ge 0.
\]
\end{lemma}

\begin{proof}
 Write \(q\) as \(q=q_{1}^{b_{1}}\!\!\dots q_{s}^{b_{s}}\), where \(q_{1},\dots,q_{s}\) are primes and \(b_{1},\dots,b_{s}\) are positive integers. Let \(a_{0}\ge 1\) be such that \(1<p<q_{i}^{a_{0}b_{i}}\) for every \(i\in\{1,\ldots,s\}\). A first step in the proof of Lemma \ref{Lemma 3.1} is to show the following
 
 \begin{fact}\label{Fact 3.2}
 Let $u\ge 1$ be a positive integer. Let $\gamma \ge 1$ be such that for every \(i\in\{1,\ldots,s\}\) and every $v\in\{1,\ldots,u\}$, $q_{i}^{\gamma}$ does not divide $p^v-1$.
  There exist integers \(N_{1}>u,\dots,N_{s}>u\) such that for every \(i\in\{1,\ldots,s\}\) and every \(N\in\N\setminus\bigcup_{j=1}^{s}N_{j}.\N\),
  \[
p^{N}\not\equiv 1\mod q_{i}^{\gamma}.
\]
 \end{fact}
 
\begin{proof}
Let \(i\in\{1,\ldots,s\}\). If \(p^{N}\not\equiv 1\mod q_{i}^{\gamma}\) for every \(N\ge 1\), then clearly \(p^{N}\not\equiv 1\mod q_{i}^{\gamma}\) for every \(N\in\N\setminus\bigcup_{j=1}^{s}N_{j}.\N\), whatever  the choice of the integers \(N_{1},\dots,N_{s}\). So we can suppose without loss of generality that there exists an integer \(N\ge 1\) such that \(p^{N}\equiv 1\mod q_{i}^{\gamma}\). Let \(N_{i}\) be the smallest such integer. Necessarily, \(N_{i}>u\), since else $q_{i}^{\gamma}$ would divide \(p^v-1\) for some $v\in\{1,\ldots,u\}$. Moreover any integer \(N\) such that \(p^{N}\equiv 1\mod q_{i}^{\gamma}\) is a multiple of \(N_{i}\). It follows that \(p^{N}\not\equiv 1\mod q_{i}^{\gamma}\) for every \(N\in\N\setminus N_{i}.\N\).
 \end{proof}
 
We apply Fact \ref{Fact 3.2} to $u=1$ and $\gamma=a_0.\max_{1\le i\le s}b_i$. Let  \(N_{1},\ldots,N_{s}\) be given by Fact \ref{Fact 3.2}.
Since \(N_{j}\ge 2\) for every \(j\in\{1,\ldots,s\}\), the set \(J:=\N\setminus \bigcup_{j=1}^{s}N_{j}.\N\) is infinite. Set $N_0=N_1\ldots N_l$. Then $I:=N_0.\N+1$ is contained in $J$.

Fix \(N\in I\). For each \(i\in\{1,\ldots,s\}\), let \(0\le a_{i,N}<\gamma\) be the largest integer such that \(q_{i}^{a_{i,N}}\) divides \(\pnm \), and write \(\pnm =q_{i}^{a_{i,N}}s_{i,N}\) for some integer \(s_{i,N}\ge 1\) with \(\textrm{gcd}(s_{i,N},q_{i})=1\). By the Fermat-Euler Theorem, there exists \(r_{i,N}\ge 1\) such that \(q_{i}^{r_{i,N}}\equiv 1\mod s_{i,N}\). 
Set \(r_{N}=r_{1,N}. r_{2,N}\dots r_{s,N}\). Then for every \(l\ge 1\) and every \(i\in\{1,\ldots,s\}\), \(q_{i}^{l.r_{N}}\equiv 1\mod s_{i,N}\), so that 
\(q_{i}^{a_{i,N}+l.r_{N}}\equiv q_{i}^{a_{i,N}}\mod (\pnm )\) for every \(i\in\{1,\ldots,s\}\). If $a$ is sufficiently large, we have \(a_{i,N}<\gamma<a.b_{i}\) for every \(i\in\{1,\ldots,s\}\), so that \(q_{i}^{a.b_{i}+l.r_{N}}\equiv q_{i}^{a.b_{i}}\mod (\pnm) \). Applying this to \(l=k.b_{i}\), \(k\ge 1\), yields that
\((q_{i}^{b_{i}})^{a+k.r_{N}}\equiv (q_{i}^{b_{i}})^{a}\mod (\pnm )\) for every \(i\in\{1,\ldots,s\}\), i.e. that \(q^{a+k.r_{N}}\equiv q^{a}\mod (\pnm )\). 
\end{proof}

\begin{proof}[Proof of Theorem \ref{Theorem 1}]
 By Lemma \ref{Lemma 3.1} above,  the sequence \((q^{n})_{n\ge 1}\) satisfies assumption (H). The proof of Theorem \ref{Theorem 2.10} combined with Lemma \ref{Lemma 3.1} shows the density of \[
S_{p,q} := \bigl\{\mu \in\pp{p,c}{}\;;\;\limsup\limits_{n\to+\infty}|\muc(q^n)|>0\bigr\}
\]
in \((\mathcal{P}_{p,c}(\T),\we)\).
\end{proof}

\begin{remark}
The proof of Theorem \ref{Theorem 1} does not make use of all the information provided by Lemma \ref{Lemma 3.1}: we apply it with $u=1$, the particular form of the set $I$ is not used, and we only need the fact that for every $N\in I$, there exist infinitely many $n$'s such that $q^{n}\equiv q^{a}\mod (\pnm )$. This additional information will be important, however, in the forthcoming proofs of Theorems \ref{Theorem 3.10} and \ref{Theorem 2}. 
\end{remark}

\subsection{A generalisation of Theorem \ref{Theorem 1}}\label{Section 3.2}
In this section, we consider  sequences \(\cnge\) of the following form: 
\(c_{n}=f_{1}(n)q_{1}^{n}+f_{2}(n)q_{2}^{n}+\cdots+f_{d}(n)q_{d}^{n}\), \(n\ge 0\), where for each \(l\in\{1,\ldots,d\}\), \(q_{l}\ge 2\) is an integer and \(f_{l}\) is a polynomial with 
coefficients in \(\Z\). This class of sequences is considered in \cite{M} and \cite{LMP}, where the following result is proved: if \(p\ge 2\) admits a prime factor \(p^{*}\) which does not divide one of the integers \(q_{i}\), \(1\le i\le d\), then any measure \(\mu \in\pp{p}{}\) which is ergodic and of positive entropy is \(\cn\)-generic. It follows that the set 
\[G'_{p,\cn}=
\bigl \{ \mu \in\pp{p,c}{}\;;\;T_{c_{n}}\mu \flw\leb\ \textrm{on a set of upper density}\ 1\bigr\} 
\]
is residual in \((\pp{p}{},\we)\) -- see Section \ref{Section 5.1} for details on this argument.
\par\smallskip
We complement this result by showing the following

\begin{theorem}\label{Theorem 3.10}
 If \(\cnge\) is a sequence of the form \(c_{n}=f_{1}(n)q_{1}^{n}+f_{2}(n)q_{2}^{n}+\cdots+f_{d}(n)q_{d}^{n}\), \(n\ge 0\), where for each \(l\in\{1,\ldots,d\}\), \(q_{l}\ge 2\) is an integer and \(f_{l}\) is a polynomial with 
coefficients in \(\Z\), then the set 
 \[G_{p,\cn}=
\bigl \{ \mu \in\pp{p,c}{}\;;\;T_{c_{n}}\mu \nflw\emph{\leb}\ \textrm{as}\ \fl{n}{+\infty\bigr\}}
\]
is residual in \((\pp{p}{},\we)\).
\end{theorem}

\begin{proof}
 Let us show that \(\cnge\) satisfies assumption (H). By Lemma \ref{Lemma 3.1}, there exist integers \(a\ge 1\)  and \(N_{l}\), \(1\le l\le d\),  such that for every 
 $l\in\{1,\ldots,d\}$ and every \(N\in I_{l}:=N_l.\N+1\), there exists an integer \(r_{l,N}\ge 1\) such that for  every \(k\ge 0\),
 \[
q_{l}^{a+k.r_{l,N}}\equiv q_{l}^{a}\mod (\pnm ).
\]
The set \(I=\bigcap_{l=1}^{d}I_{l}\) is infinite. If we set, for each $N\in I$, \(r_{N}=r_{1,N}\dots r_{d,N}\), we get that for every \(N\in I\) and every \(k\ge 0\),
\begin{equation}\label{Equation 1}
 q_{l}^{a+k.r_{N}}\equiv q_{l}^{a}\mod (\pnm )\quad \textrm{for every}\ l\in\{1,\ldots,d\}.
\end{equation}
For each $l\in\{1,\ldots,d\}$, write the polynomial \(f_{l}\) as
\[
f_{l}(x)=\sum_{j=0}^{\Delta _{l}}b_{j}^{(l)}x^{j}, \quad \textrm{where}\ b_{j}^{(l)}\in\Z\ \textrm{for every}\ j\in\{0,\ldots, \Delta _{l}\}.
\]
For every $N\in I$ and every integer
\(k'\ge 0\), we have
\begin{align}
 f_{l}\bigl ( a+k'(\pnm )r_{N}\bigr)&=\sum_{j=0}^{\Delta_{l}}b_{j}^{(l)}\bigl (a+k'(\pnm )r_{N} \bigr)^{j}\notag 
 \intertext{and}
 \bigl (a+k'(\pnm )r_{N} \bigr)^{j} &\equiv a^{j}\mod (\pnm )\quad \textrm{for every}\ j\ge 0\notag.
 \intertext{Hence}
 f_{l}\bigl ( a+k'(\pnm )r_{N}\bigr)&\equiv f_{l}(a)\mod (\pnm )\quad \textrm{for every}\ l\in\{1,\ldots,d\}.\label{Equation 2}\\
 \intertext{Putting together (\ref{Equation 1}) and (\ref{Equation 2}) yields that for every \(N\in I\),}
 c_{a+k'(\pnm)r_{,N}}&\equiv c_{a}\mod (\pnm)\quad \textrm{for every}\ k'\ge 0\notag,
\end{align}
\par\medskip
\noindent which implies that assumption (H) is true. Theorem \ref{Theorem 3.10} thus follows from Theorem \ref{Theorem 2.10}.
 \end{proof}
\begin{remark}\label{Remark 3.11}
 We notice that if \(c_{n}=f(n)\) for some polynomial \(f\in\Z[X]\), then the set 
 \[
\bigl \{ \mu \in \pp{p,c}{}\;;\;T_{c_{n}}\mu \nflw\leb\bigr\} 
\]
is also residual in \((\pp{p}{},\we)\). Remark that in this case, the sequence \((c_{n}x)_{n\ge 0}\) is uniformly distributed mod 1 for every \(x\in\R\setminus\Q\), and hence
\begin{equation}\label{Equation 3}
  \fr{}\sum_{n=0}^{N-1}\exp(2i\pi hc_{n}x) \to 0 \quad \textrm{as } \quad N \to \infty\ \quad \textrm{for every}\ h\in\Z\setminus\{0\}.
\end{equation}
If \(\mu \) belongs to \(\pp{p,c}{}\), integrating (\ref{Equation 3}) with respect to the measure $\mu$ yields that 
\begin{align*}
  \fr{}\sum_{n=0}^{N-1}\muc(hc_{n})& \to 0 \quad \textrm{for every}\ h\in\Z\setminus\{0\},
 \quad \textrm{i.e.}\quad \fr{}\sum_{n=0}^{N-1}T_{c_{n}}\mu \flw\leb.
\end{align*}
\end{remark}

\subsection{Further examples: proof of Theorem \ref{Theorem 2}}
Let \(\cnge\) be defined by a linear recursion: there exist \(L\ge 1\) and coefficients 
\(a_{1},\dots,a_{L}\) in \(\Z\) with \(a_{L}\neq 0\) such that
\[
c_{n}=a_{1}c_{n-1}+a_{2}c_{n-2}+\cdots+a_{L}c_{n-L}\quad \textrm{for every}\ n\ge L.
\]
Let \(p\ge 2\). Meiri introduces in \cite{M} the following two assumptions on the sequence \(\cnge\):
\begin{enumerate}
 \item [(a)] \(\cnge\) has no non-constant arithmetic subsequence;
 \item [(b)] \(a_{L}\) and \(p\) are relatively prime.
\end{enumerate}
It is observed in \cite{M}*{Prop. 5.1} that assumption (a) is satisfied as soon as the following property holds: 
\begin{enumerate}
 \item [(a')] if \(\lambda _{1},\dots,\lambda _{L'}\), with $1\le L'\le L$, are the distinct roots of the recursion polynomial \(p(x)=x^{L}-\sum_{l=1}^{L}a_{l}x^{L-l}\), then none of the quantities  \(\lambda _{i}\) and \(\lambda _{i}/\lambda _{j}\), \(1\le i<j\le L'\) is a root of unity.
\end{enumerate}
If assumptions (a) and (b) are satisfied, any measure \(\mu \in\pp{p}{}\) which is ergodic and of positive entropy is such that
\(T_{c_{n}}\mu \flw{}\leb\) as \(\fl{n}{+\infty}\) (this is a consequence of \cite{M}{Th. 5.2}). In particular, the set 
\[G'_{p,\cn}=
\bigl \{ \mu \in\pp{p,c}{}\;;\;T_{c_{n}}\mu \flw\leb\quad \textrm{along a set of upper density}\ 1\bigr\} 
\]
is residual in \((\pp{p}{},\we)\). See also \cite{LMP}*{Thm. 4.3}.
\par\medskip
Theorem \ref{Theorem 2} complements these results by showing that under the sole assumption (b), the set  
\[G_{p,\cn}=
\bigl \{ \mu \in\pp{p,c}{}\;;\;T_{c_{n}}\mu \nflw\leb\bigr\}
\]
is residual in \((\pp{p}{},\we)\).

\begin{proof}[Proof of Theorem \ref{Theorem 2}]
 Consider the matrix \(A\in M_{L}(\Z)\) given by
 \[
A=\begin{pmatrix}
a_{1}&a_{2}&\dots&\dots&a_{L}\\
1&0&\dots&\dots&0\\
0&1&\ddots&&0\\
\vdots&\ddots&\ddots&\ddots&\vdots\\
0&\dots&0&1&0
\end{pmatrix}.
\]
Setting \(C_{n}={}^{\textrm{T}}\!
\begin{pmatrix}
c_{n}&c_{n-1}&\dots&c_{n-L+1} 
\end{pmatrix}
\) for every \(n\ge L-1\), we have \(C_{n+1}=AC_{n}\) for every \(n\ge L-1\). Since \(\det A=(-1)^{L+1}a_{L}\) and \(a_{L}\neq 0\), \(A\) is invertible as a matrix of \(M_{L}(\Q)\), and 
\[
A^{-1}=\dfrac{(-1)^{L+1}}{a_{L}}\,\textrm{adj}\,(A)
\]
where \(\textrm{adj}\,(A)\), the adjoint (or adjugate) of \(A\), is the transpose of the matrix of the cofactors of \(A\). Observe that  \(\textrm{adj}\,(A)\in M_{L}(\Z)\).
\par\smallskip
Let us decompose the integer \(a_{L}\) as \(a_{L}=q_{1}^{b_{1}}\dots q_{s}^{b_{s}}\), where  \(q_{i}\) is prime and \(b_{i}\ge 1\) for every $i\in\{1,\ldots,s\}$. By Fact \ref{Fact 3.2}, there exist \(\gamma\ge 1\) and an infinite subset \(I\) of \(\N\) such that for every $N\in I$, \(p^{N}\not\equiv 1\mod q_{i}^{\gamma}\) for every $i\in\{1,\ldots,s\}$. Hence, for every \(N\in I\), $p^{N}-1$ can be written as \(p^{N}-1=q_{1}^{\beta _{1,N}}\!\!\!\dots\, q_{s}^{\beta _{s,N}}r_{N}\), where \(0\le \beta _{i,N}< \gamma\) and \(\textrm{gcd}(q_{i},r_{N})=1\) for each $i\in\{1,\ldots,s\}$. Since the prime factors of \(a_{L}\) are exactly the \(q_{i}\)'s, it follows that \(a_{L}\) and \(r_{N}\) are relatively prime, and hence that \(a_{L}\) is invertible modulo \(r_{N}\): there exists an integer \(d_{N}\) with \(0\le d_{N}<r_{N}\) such that \(a_{L}.d_{N}\equiv 1\mod r_N\).
Setting \(B_{N}=(-1)^{L+1}d_{N}\,\textrm{adj}(A)\), we observe that \(B_{N}\in M_{L}(\Z)\) and that \(AB_{N}\equiv B_{N}A\equiv I \mod r_{N}\). So \(A\) is invertible modulo \(r_{N}\), and its inverse is \(B_{N}\).
\par\smallskip
Consider now the set of matrices in \(M_{L}(\Z)\) consisting of all powers \(A^{n}\), \(n\ge 0\), of \(A\), taken modulo \(r_{N}\). This set being finite, there exist two integers  \(0\le n_{1,N}<n_{2,N}\) such that \(A^{n_{1,N}}\equiv A^{n_{2,N}}\mod r_{N}\). 
Setting \(n_{N}=n_{2,N}-n_{1,N}\), \(A^{n_{N}}\equiv I\mod r_{N}\), and thus the sequence \((C_{n})\) taken modulo \(r_{N}\) is periodic, with period \(n_{N}\). It follows that the sequence \(\cn\) itself taken modulo \(r_{N}\) is periodic of period \(n_{N}\), so that, in particular, \(c_{jn_{N}}\equiv c_{0}\mod r_{N}\) for every \(j\ge 0\). Setting 
\(h_{N}=q_{1}^{\beta _{1,N}}\!\!\!\dots\,q_{s}^{\beta _{s,N}}\) and remembering that \(\pnm=h_{N}.r_{N}\), we obtain that 
\(h_{N}.c_{jn_{N}}\equiv h_{N}c_0\mod (\pnm)\) for every \(j\ge 0\). Since \(0\le \beta _{i,N}<\gamma\) for every $i\in\{1,\ldots,s\}$, the set \(\{h_{N}\;;\;N\in I\}\) is finite and consists of non-zero integers. Assumption (H) is satisfied, and the proof is concluded as usual thanks to Theorem \ref{Theorem 2.10}.\end{proof}

\section{The multidimensional case: proof of Theorem \ref{Theorem 3}}\label{Section 4}
In this section, \(d\ge 2\) is an integer, and \(A,B\in M_{d}(\Z)\) are two \(d\times d\) matrices with integer coefficients such that $\det A\neq 0$ and $\det B\neq 0$. The matrix \(A\) is supposed to be similar in \(M_{d}(\C)\) to a diagonal matrix \(D\) whose diagonal coefficients \(\lambda _{1},\dots,\lambda _{d}\) are not of modulus \(1\). Let \(P\in GL_{d}(\C)\) be such that \(A=PDP^{-1}\). The matrix \(B\) is supposed to be invertible in \(M_{d}(\C)\), i.e. \(\det B\neq 0\). Theorem \ref{Theorem 3} states that the set 
\[
\gro{A,B}{}=\bigl \{ \mu \in\pp{A,c}{d}\;;\;T_{B^{n}}\mu \nflw\leb\bigr \}
\]
is residual in \((\pp{A}{d},\we)\). The proof of Theorem \ref{Theorem 3} follows the same structure as those of Theorems \ref{Theorem 2.10} and \ref{Theorem 1}, but certain technical difficulties that come with the multidimensional setting must be overcome.
\par\smallskip
We begin by proving an analogue of Theorem \ref{Theorem 2.1}.

\subsection{Some dense classes of discrete measures in the set of A-invariant measures}
Let \((\nk )_{k\ge 1}\) be a strictly increasing sequence of integers. Consider the set 
\[
\cc_{A,(\nk )}=\bigl \{\xx\in \T^{d}\;;\;A^{\nk }\xx=\xx\quad \textrm{for some}\ k\ge 1 \bigr \}
\]
which consists of periodic points for \(T_{A}\) having a period within the set 
\(\{\nk \;;\;k\ge 1\}\). For each \(\xx\in\cc_{A,(\nk )}\), let \(\mu _{\xx}\) be the measure defined by 
\[
\mu _{\xx}=\fr{k}\sum_{j=0}^{\nk -1}\de{A^{j}\xx}.
\]
It is a discrete \(T_{A}\)-invariant probability measure on \(\T^{d}\). Set 
\[
\dd_{A,(\nk )}=\bigl \{ \mu _{\xx}\;;\;\xx\in\cc_{A,(\nk )}\bigr \}.
\]
Taking inspiration from Theorem \ref{Theorem 2.1}, we would like to show that the set \(\dd_{A,(\nk )}\) is dense in \((\pp{A}{d},\we)\). If \((\T^{d},T_{A})\) has the periodic specification property, this is an immediate consequence of Remark \ref{Remark 2.11}. However, \(T_{A}\) is known to have the periodic specification property only in the case where \(A\) is an hyperbolic automorphism of \(\T^{d}\), i.e. \(\det A=\pm 1\) and \(A\) has no eigenvalue of modulus \(1\). Since \(A\) is not assumed here to be an automorphism of \(\T^{d}\), we need to take a different route. It will lead to the following weaker result, which is fortunately sufficient for our purposes:

\begin{theorem}\label{Theorem 4.10} 
 The convex hull of the set \(\dd_{A,(\nk )}\) is dense in \((\pp{A}{d},\we)\).
\end{theorem}
\begin{proof}
 Denote by \(\ff_{A,(\nk )}\) the \(\we\)-closure in \(\pp{A}{d}\) of the convex hull of  
 \(\dd_{A,(\nk )}\). This is a \(\we\)-closed convex subset of \(\pp{A}{d}\), and also of the Banach space \(\mathcal{M}(\td )\) of complex measures on \(\td \), endowed with the norm \(||\mu ||:=|\mu |(\td )\). This space \(\mathcal{M}(\td )\) is the dual space of \((C(\td ), ||\,.\,||_{\infty,\td})\), the space of continuous functions on \(\td\).
 \par\smallskip
 Our aim is to show that \(\ff_{A,(\nk)}=\pp{A}{d}\). Suppose that it is not the case, and that there exists a measure \(\mu _{0}\) belonging to \(\pp{A}{d}\setminus\ff_{A,(\nk)}\).
 \par\smallskip
 Applying the Hahn-Banach Theorem in the locally convex space \((\mathcal{M}(\T^d),w^{*})\), we obtain that there exists a \(w^{*}\)-continuous linear functional \(L:\mathcal{M}(\T^d)\longrightarrow \C\), as well as real numbers \(\gamma_{1} <\gamma _{2}\) such that 
\[
\Re e(L(\mu ))\le \gamma _{1}<\gamma _{2}\le\Re e(L(\mu _{0}))
\]
 for every \(\mu \in \ff_{A,(\nk )}\). Since any \(w^{*}\)-continuous functional on \(\mathcal{M}(\T^d)=C(\T^d)^{*}\) acts as integration against an element of \(C(\T^d)\), there exists a function \(f\in C(\T^d)\) such that 
 \[
\Re e\int_{\T}f\,d\mu \le \gamma _{1}<\gamma _{2}\le\Re e\int_{\T^d}f\,d\mu _{0}
\]
for every \(\mu \in \ff_{A,(\nk )}\). The measures \(\mu \) and \(\mu _{0}\) being nonnegative, replacing \(f\) by its real part we  can assume that \(f\) is real-valued, and thus that
\begin{equation}\label{Equation 1bis}
 \int_{\T}f\,d\mu \le \gamma _{1}<\gamma _{2}\le\int_{\T^d}f\,d\mu _{0}\quad \textrm{for every}\ \mu \in \ff_{A,(\nk )}.
\end{equation}
Moreover, it is possible to assume that \(f\) is a Lipschitz map on \(\T^d=\R^d/\Z^d\), endowed with the distance induced by the sup norm $||\, .\,||_{\infty,\R^d}$ on $\R^d$.
We thus suppose that there exists a constant \(C>0\) such that 
\begin{equation}\label{Lipschitz}
 |f(\xx_{1})-f(\xx_{2})|\le C\,\inf\left\{||\xx_{1}-\xx_{2}-\pmb{l}||_{\infty,\R^d}\;;\; \pmb{l} \in\Z^d\right\} \quad\textrm{for every }\xx_{1},\,\xx_{2}\in\T^d.
\end{equation}

\par\smallskip
For any integer \(k\ge 1\) and any element \(\xx_{k}\) of \(\T^d\) such that \((A^{{N_{k}}}-I)\xx_k=0\), the measure
\(\mu _{\xx_k}=\frac{1}{N_{k}}\sum_{j=0}^{N_{k}-1}\delta _{\{A^{{j}}\xx_k\}}\) belongs to \(\ff_{A,(\nk )}\). Applying (\ref{Equation 1bis}) to this measure yields that
\begin{equation}\label{Equation 2bis}
 \dfrac{1}{N_{k}}\sum_{j=0}^{N_{k}-1}f(A^{{j}}\xx_k)\le\gamma _{1}<\gamma _{2}\le\int_{\T^d}f\,d\mu _{0}.
\end{equation}
Let now \(\xx\) be an arbitrary element of \(\td\), and let \(k\ge 1\). Consider the vector \(\yy_k=(A^{\nk}-I)\,\xx\), seen as an element of \(\R^{d}\) (and not as an element of \(\td\)). There exists 
 \(\llk \in\Z^{d}\) such that \( ||\yy_{k}-\llk ||_{\infty,\R^d}\le 1\).
Recalling that \(A=PDP^{-1}\), with \(D=\textrm{diag}(\lambda _{1},\dots,\lambda _{d})\), we thus have 
\begin{align}
 \bigl|\bigl|\,P(D^{\nk}-I)P^{-1}\xx-\llk \,\bigr|\bigr|_{\infty,\R^d}&\le 1,\notag
 \intertext{so that}
 \bigl|\bigl|\,(D^{\nk}-I)P^{-1}\xx-P^{-1}\llk \,\bigr|\bigr|_{\infty,\C^d}&\le ||P^{-1}||_{\infty}\label{(12)}
\end{align}
where \(||P^{-1}||_{\infty}\) is the norm of the matrix \(P^{-1}\) seen as an endomorphism of \((\C^{d},||\,.\,||_{\infty,\C^d})\).
The inequality (\ref{(12)}) means exactly that
\begin{equation}\label{(13)}
 \sup_{1\le i\le d} \bigl | \,(\lambda _{i}^{\nk}-1)\,(P^{-1}\xx)_{i}-(P^{-1}\llk )_{i}\bigr|\le||P^{-1}||_{\infty}.
\end{equation}
Since no eigenvalue of \(A\) belongs to the unit circle, $A^{\nk}-I$ is invertible in $M_d(\R)$ and it is legitimate to set \(\zzk =(A^{\nk}-I)^{-1}\llk \in\R^d\). Let $\xxk$ be
the corresponding element of \(\td\), obtained by taking mod $1$ all the coordinates of $\zzk$. 
Then \(\xxk \) belongs to \(\cc_{A,(\nk)}\), with \((A^{\nk}-I)\xxk =0\) in \(\td\). Also,
\(\zzk=P(D^{\nk}-I)^{-1}P^{-1}\,\llk \), so that
\[
P^{-1}\zzk=\Bigl (\dfrac{1}{\lambda _{i}^{\nk}-1}\cdot\bigl (P^{-1}\llk \bigr)_{i}  \Bigr)_{1\le i\le d}.
\]
It follows from (\ref{(13)}) that for every \(i\in\{1\ldots,d\}\),
\begin{equation}\label{(14)}
 \bigl | (P^{-1}\xx)_{i}-(P^{-1}\zzk)_{i}\bigr|\le\dfrac{||P^{-1}||_{\infty}}{|\lambda _{i}^{\nk}-1|}\cdot
\end{equation}
By (\ref{Lipschitz}),
\begin{align*}
 \bigl|\,f(A^{j}\xx)-f(A^{j}\xxk)\,\bigr|&\le
 C\,\inf\left\{||A^{j}\xx-A^{j}\xxk- \pmb{l}||_{\infty,\R^d}\;;\; \pmb{l} \in\Z^d\right\}\\
 &\le
 C\, \bigl| \bigl| A^{j}\xx-A^{j}\zzk \bigr|  \bigr|_{\infty,\R^d} \\
 &\le C \,||P||_{\infty}\,\bigl| \bigl| \,D^{j}(P^{-1}\xx-P^{-1}\zzk)\, \bigr|  \bigr|_{\infty,\C^d}\\
 &=C \,||P||_{\infty}\,\sup_{1\le i\le d}\,|\lambda _{i}^{j}|\,.\,\bigl |(P^{-1}\xx-P^{-1}\zzk)_{i} \bigr| \\
 &\le C \,||P||_{\infty}\,\sum_{i=1}^{d}\,|\lambda _{i}^{j}|\,.\,\bigl |(P^{-1}\xx-P^{-1}\zzk)_{i} \bigr|.
\end{align*}
Plugging into (\ref{(14)}) yields that
\[
\bigl|\,f(A^{j}\xx)-f(A^{j}\xxk)\,\bigr|\le C \,||P||_{\infty}\,.\,||P^{-1}||_{\infty}\,
\sum_{i=1}^{d}\dfrac{|\lambda _{i}|^{j}}{|\lambda _{i}^{\nk}-1|}\cdot
\]
Hence 
\begin{align*}
 \Bigl |\fr{k}\sum_{j=0}^{\nk-1}f(A^{j}\xx)-\fr{k}\sum_{j=0}^{\nk-1}f(A^{j}\xxk)\Bigr|
 &\le C \,||P||_{\infty}\,||P^{-1}||_{\infty}\,\fr{k}\sum_{i=1}^{d}\dfrac{1}{|\lambda _{i}^{\nk}-1|}\sum_{j=0}^{\nk-1}|\lambda _{i}|^{j}\\
 &\le C \,||P||_{\infty}\,||P^{-1}||_{\infty}\,\fr{k}\sum_{i=1}^{d} \dfrac{|\lambda _{i}|^{\nk}-1}{|\lambda _{i}^{\nk}-1|\,(|\lambda _{i}|-1)}\cdot
\end{align*}
Notice that \(|\lambda _{i}|\neq 1\) for all \(i=1\ldots d\). Observe that 
$(|\lambda _{i}|^{\nk}-1)/|\lambda _{i}^{\nk}-1|\longrightarrow 1$ as ${k}\longrightarrow{+\infty}$ if $|\lambda _{i}|>1$, while
  $({|\lambda _{i}|^{\nk}-1)/|\lambda _{i}^{\nk}-1|\longrightarrow-1}$ as ${k}\longrightarrow{+\infty}$ if $|\lambda _{i}|<1$.
We obtain the existence of a positive constant \(C'\) such that 
\[
\sup_{k\ge 1}\,\,\sum_{i=1}^{d} \dfrac{|\lambda _{i}|^{\nk}-1}{|\lambda _{i}^{\nk}-1|\,(|\lambda _{i}|-1)}\le C'.
\]
Thus there exists \(C''>0\) such that 
\begin{equation}\label{(15)}
\Bigl |\fr{k}\sum_{j=0}^{\nk-1}f(A^{j}\xx)-\fr{k}\sum_{j=0}^{\nk-1}f(A^{j}\xxk)\Bigr|\le
\dfrac{C''}{N_k}\quad \textrm{for every}\ k\ge 1.
\end{equation}
The right hand bound in (\ref{(15)}) tends to \(0\) as \(k\) tends to infinity. Combining this with the fact that inequalities (\ref{Equation 2bis}) and (\ref{(15)}) hold true for every \(k\ge 1\), we obtain that
\begin{equation}\label{Equation 5}
 \limsup_{k\to+\infty}\dfrac{1}{N_{k}}\sum_{j=0}^{N_{k}-1}f(A^{j}\xx )\le \gamma _{1}<\gamma _{2}\le\int_{\T^d}f\,d\mu _{0}\quad \textrm{for every } \xx \in\T^d.
\end{equation}
Let \(\varepsilon >0\) be such that \(\gamma _{1}<\gamma _{2}-\varepsilon \). Applying the Ergodic Decomposition Theorem to the measure \(\mu _{0}\) yields the existence of an ergodic \(T_{A}\)-invariant measure \(\nu _{0}\) on \(\T^d\) such that 
\[
\int_{\T^d} f\,d\nu _{0}\ge\int_{\T^d}f\,d\mu _{0}-\varepsilon .
\]
It then follows from (\ref{Equation 5}) that
\begin{equation}\label{Equation 6}
 \limsup_{k\to+\infty}\dfrac{1}{N_{k}}\sum_{j=0}^{N_{k}-1}f(A^{j}\xx)\le \gamma _{1}<\gamma _{2}-\varepsilon \le\int_{\T^d}f\,d\nu _{0}
\end{equation}
for every \(\xx \in\T^d\).
\par\medskip
But since the measure \(\nu _{0}\) is ergodic, the Birkhoff Pointwise Ergodic Theorem implies that 
\[
\limsup_{k\to+\infty}\dfrac{1}{N_{k}}\sum_{j=0}^{N_{k}-1}f(A^{j}\xx )=\int_{\T^d}f\,d\nu _{0}
\]
for \(\nu _{0}\)-almost every \(\xx\in\T^d\), which contradicts (\ref{Equation 6}). So the initial assumption that the set \(\pp{A}{d}\setminus \ff_{A,(\nk )}\) is non-empty cannot hold, and Theorem \ref{Theorem 4.10} is proved
\end{proof}

\subsection{Proof of Theorem \ref{Theorem 3}} 
The proof of Theorem \ref{Theorem 3} is similar in spirit to that of Theorem \ref{Theorem 1}. Of course, assumption (H) and Theorem \ref{Theorem 2.10} are not available anymore, and they have to be replaced by the following analogue of Fact \ref{Fact 3.2}:
\begin{lemma}\label{Lemma 4.11}
 Let \(A\in M_{d}(\Z)\) with \(\det A\neq 0\) and \(\det(A-I)\neq 0\), and let \(p_{1},\dots,p_{s}\ge 2\) be prime numbers such that \(\emph{gcd\,}(p_{i},\det A)=1\) for every $i\in\{1,\ldots,s\}$. There exist an infinite subset \(I\) of \(\N\) and an integer \(\gamma \ge 1\) such that for every $i\in\{1,\ldots,s\}$ and every \(N\in I\), 
 \[
\det\,\bigl (A^{N}-I \bigr)\not\equiv 0\mod p_{i}^{\gamma }. 
\]
\end{lemma}
\begin{proof}[Proof of Lemma \ref{Lemma 4.11}]
Since \(\det (A-I)\neq 0\), there exists \(\gamma \ge 1\) such that for every $i\in\{1,\ldots,s\}$, \(p_{i}^{\gamma }\) does not divide \(\det(A-I)\). Since \(\gc (p_{i},\det A)=1\) and \(p_{i}\) is prime, \(\gc (p_{i}^{\gamma },\det A)=1\) as well, and \(A\) is invertible modulo \( p_{i}^{\gamma }\). Proceeding as in the proof of Theorem \ref{Theorem 2}, we obtain an integer \(n_{i}\ge 2\) such that \(A^{n_{i}}\equiv I\mod p_{i}^{\gamma }\), and hence \(A^{ln_{i}}\equiv I\mod p_{i}^{\gamma }\) for every $l\ge 1$. Setting \(n_{0}=n_{1}\dots n_{s}\), we have \(A^{ln_{0}}\equiv I\mod p_{i}^{\gamma }\) for every \(l\ge 1\) and every $i\in\{1,\ldots,s\}$. Thus 
\(A^{ln_{0}+1}-I\equiv A-I\mod p_{i}^{\gamma }\) and \(\det(A^{ln_{0}+1}-I)\equiv \det(A-I)\mod p_{i}^{\gamma }\). Since \(\det(A-I)\not\equiv 0\mod p_{i}^{\gamma }\), we have \(\det(A^{ln_{0}+1}-I)\not\equiv  0\mod p_{i}^{\gamma }\) for every \(l\ge 1\) and every $i\in\{1,\ldots,s\}$, and the lemma follows by setting \(I=n_{0}.\N+1\).
\end{proof}

Our aim is now to show that under the assumptions of Theorem \ref{Theorem 3}, the following fact holds:
\begin{fact}\label{Fact 4.12}
 Suppose that \(A,B\in M_{d}(\Z)\) satisfy assumption (a), and either assumption (b) or (b') of Theorem \ref{Theorem 3}. There exist a strictly increasing sequence \((\nk)_{k\ge 1}\) of integers and a finite subset \(F\) of \(\Z\setminus\{0\}\) such that, for every \(k\ge 1\), the integers \(q_{k}:=\det (A^{\nk}-I)\) can be decomposed as \(q_{k}=h_{k}.r_{k}\), where \(h_{k}\in F\), $r_k\ge 1$, and \(\gc(r_{k},\det B)=1\).
\end{fact}

\begin{proof}[Proof of Fact \ref{Fact 4.12}] Recall that since $A$ has no eigenvalue of modulus $1$, $A^{N_k}-I$ is invertible in $M_d(\C)$, and hence \(q_{k}:=\det (A^{\nk}-I)\neq 0\).
If \(\det B=\pm 1\), it suffices to choose \(\nk=k\), \(k\ge 1\), and \(F=\{\pm 1\}\). So we suppose without loss of generality that \(|\det B|\ge 2\). We decompose \(\det B\) as 
\(\det B=\varepsilon\, p_{1}^{b_{1}}\dots p_{s}^{b_{s}}\), where $\varepsilon=\pm 1$, \(b_{i}\ge 1\) and \(p_{i}\) is a prime number for every $i\in\{1,\ldots,s\}$. We now treat separately two cases:
\par\smallskip
\noindent
\emph{Case} 1: assumption (b) is satisfied, i.e. \(\gc (\det A,\det B)=1\). In this case
\(
 \gc(p_{i},\det A)=1
\) for every $i\in\{1,\ldots,s\}$, and Lemma \ref{Lemma 4.11} applies: there exist \(\gamma \ge 1\) and an infinite set \(I\subseteq\N\) such that for every $i\in\{1,\ldots,s\}$ and every \(N\in I\), \(p_{i}^{\gamma }\) does not divide \(\det(A^{N}-I)\). We enumerate the set \(I\) as a strictly increasing sequence \((\nk)_{k\ge 1}\), and for each \(k\ge 1\) we decompose \(q_{k}=\det (A^{\nk}-I)\) as \(q_{k}=\varepsilon_k\,p_{1}^{a_{1,k}}\dots p_{s}^{a_{s,k}}r_k\), where $\varepsilon_k=\pm 1$, \(0\le a_{i,k}<\gamma \) and \(\gc(r_{k},p_{i})=1\) for each $i\in\{1,\ldots,s\}$. Setting 
\[
F=\bigl \{\pm  p_{1}^{a_{1}}\dots p_{s}^{a_{s}}\;;\;0\le a_{i}< \gamma,\;\;i=1\ldots s  \bigr \}
\]
yields the conclusion of Fact \ref{Fact 4.12} in this case.
\par\smallskip
\noindent
\emph{Case} 2: assumption (b') is satisfied. Let \(a_{1},\dots,a_{d}\) be the diagonal coefficients of \(A\), which belong to $\Z\setminus\{0\}$. For every \(N\ge 1\), 
\(\det(A^{N}-I)=\prod_{l=1}^{d}(a_{l}^{N}-1)\). By Fact \ref{Fact 3.2} applied with $u=2$, there exist for each $l\in\{1,\ldots,d\}$ integers \(N_{1,l}\ge 3,\dots,N_{s,l}\ge 3\) as well as \(\gamma_{l} \ge 1\) such that for every $i\in\{1,\ldots,s\}$ and every \(N\in \N\setminus\bigcup_{i=1}^{s}N_{i,l}\,.\,\N\),	
\[
  |a_{l}|^{N}\not\equiv 1\mod p_{i}^{\gamma _{l}}.
\]
Since the integers $N_{i,l}$ are all greater or equal to $3$, the set \(J=\N\setminus\bigcup_{l=1}^{d}\bigcup_{i=1}^{s}N_{i,l}\,.\,\N\) contains an infinite subset $J'$ consisting of \emph{even} integers. Let $I=\{M\ge 1\;;\; 2M\in J'\}$.
For every $i\in\{1,\ldots,s\}$, every $l\in\{1,\ldots,d\}$ and every \(M\in I\),
\[
  a_{l}^{M}\not\equiv 1\mod p_{i}^{\gamma _{l}}.
\]
Setting \(\gamma_{0} =\max_{1\le l\le d}\gamma _{l}\), we have thus
\[
a_{l}^{M}\not\equiv 1\mod p_{i}^{\gamma _{0}}
\]
for every $l\in\{1,\ldots,d\}$, $i\in\{1,\ldots,s\}$ and $M\in I$.
We now set \(\gamma :=d\gamma _{0}\). Then \(p_{i}^{\gamma }\) cannot divide the product 
\(\prod_{l=1}^{d}(a_{l}^{M}-1)\), since else \(p_{i}^{\gamma_0 }\) would divide one of the terms \(a_{l}^{M}-1\), \(1\le l\le d\). Hence \(\det(A^{M}-I)\not\equiv 0\mod p_{i}^{\gamma }\) for every $i\in\{1,\ldots,s\}$,  and we conclude the proof as in the first case.
\end{proof}

We now use the notation from Fact \ref{Fact 4.12}. Since, for each \(k\ge 1\), \(r_{k}\) and \(\det B\) are relatively prime, \(B\) is invertible modulo \(r_{k}\), and there exists an integer \(m_{k}\ge 1\) such that \(B^{m_{k}}\equiv I\mod r_{k}\) (see the proof of Theorem \ref{Theorem 2}). Hence \(h_{k}\,B^{m_{k}}\equiv h_{k}\,I\mod q_{k}\) for every \(k\ge 1\). If we define \(h_{0}\)
to be the product of all the elements of the finite set $F$, it follows that
\[
h_{0}.B^{m_{k}}\equiv h_{0}.I\mod q_{k}\quad \textrm{for every}\ k\ge 1.
\]
Recall that given a measure \(\mu \in\pp{}{d}\) and a \(d\)-tuple \(\nn=(n_{1},\dots,n_{d})\in\Z^{d}\), the \(\nn\)-th Fourier coefficient of the measure \(\mu \)
is defined as
\[
\muc(\nn)=\int_{\T^d}e^{2i\pi \pss{\nn}{\yy}}d\mu (\yy), \quad \textrm{where}\ \pss{\nn}{\yy}=\sum_{i=1}^{d}n_{i}y_{i}.
\]
Write \(\hh_{0}=(h_{0},\dots,h_{0})\). Here is now an analogue of Fact \ref{Fact 3.3} in our multidimensional setting:
\begin{fact}\label{Fact 4.13}
 Let \(\xx\in\td\) be such that \((A^{\nk}-I)\xx=0\) in \(\td\) for some \(k\ge 1\). Then
 \[
\wh{B^{lm_{k}}{\mu _{\xx}}}(\hh_{0})=\muc_{\xx}(\hh_{0})\quad \textrm{for every integer}\ l\ge 1.
\]
\end{fact}
\begin{proof}[Proof of Fact \ref{Fact 4.13}]
Recall that
\[
\mu _{\xx}=\fr{k}\sum_{j=0}^{\nk -1}\de{A^{j}\xx}.
\]
 For every \(n\in\Z\), 
 \begin{align*}
  \wh{B^{n}{\mu _{\xx}}}(\hh_{0})&=\fr{k}\sum_{j=0}^{\nk-1}e^{2i\pi \pss{\hh_{0}}{B^{n}A^{j}\xx}}=\fr{k}\sum_{j=0}^{\nk-1}e^{2i\pi \pss{\pmb{1}}{h_{0}B^{n}A^{j}\xx}}\intertext{while}
  \wh{\mu _{\xx}}(\hh_{0})&=\fr{k}\sum_{j=0}^{\nk-1}e^{2i\pi \pss{\pmb{1}}{h_{0}A^{j}\xx}}
 \end{align*}
 where \(\pmb{1}=(1,\dots,1)\).
 Since \((A^{\nk}-I)\xx=0\) in \(\td\), there exists \(\lle_{k}\in\Z^{d}\) such that \((A^{\nk}-I)\xx=\lle_{k}\), the equality being this time in \(\R^{d}\), so that \(\xx=\dfrac{1}{q_{k}}\,\textrm{adj}\,(A^{\nk}-I)\lle_{k}\).
 \par 
We know that \(h_{0}\,B^{m_{k}}\equiv h_{0}\,I\ \textrm{mod}\ q_{k}\), so that \(h_{0}\,B^{lm_{k}}\equiv h_{0}\,I\ \textrm{mod}\ q_{k}\) for every \(l\ge 1\). This means that there exists a matrix \(C_{k,l}\in M_{d}(\Z)\) such that \(h_{0}.B^{lm_{k}}=h_{0}I+q_{k}.C_{l,k}\). Hence for every \(j\in\{0,\ldots, \nk\}\), we have
\[
h_{0}\,B^{lm_{k}}A^{j}\xx=h_{0}\,A^{j}\xx+q_{k}\,C_{k,l}A^{j}\xx.
\]
Since \(\xx=\dfrac{1}{q_{k}}\,\textrm{adj}\,(A^{\nk}-I)\lle_{k}\) with \(\lle_{k}\in\Z^{d}\), the vector \(q_{k}C_{k,l}A^{j}\xx\) belongs to \(\Z^{d}\), and thus \(h_{0}\,B^{lm_{k}}A^{j}\xx=h_{0}\,A^{j}\xx\) in \(\td\). It follows that
\[
\wh{B^{lm_{k}}\mu _{\xx}}(\hh_{0})=\muc_{\xx}(\hh_{0})\quad \textrm{for every}\ l\ge 1,
\]
and Fact \ref{Fact 4.13} is proved.
\end{proof}

A direct consequence of Fact \ref{Fact 4.13} is that \(B^{n}\mu _{\xx}\nflw\leb_d\) when  \(\fl{n}{+\infty}\) as soon as \(\wh{\mu} _{\xx}(\hh_{0})\neq 0\).
Consider, for each \(0<\gamma <1\), the set 
\[
\gro{A,B}^{\,\gamma }=\bigl \{ \mu \in\pp{A}{d}\;;\;\muc(\hh_{0})\neq 0\ \textrm{and}\ \forall\,n_{0}\ge 1,\ \exists\,n\ge n_{0},\ |\wh{B^{n}\mu }(\hh_{0})|>\gamma |\muc(\hh_{0})| \bigr\} 
\]
which is clearly a \(G_{\delta }\) subset of \((\pp{A}{d},\we)\). In order to prove that it is dense, we proceed as in the proof of Lemma \ref{Lemma 2.3}, but using Theorem \ref{Theorem 4.10} instead of Theorem \ref{Theorem 2.10}. Let \(\mathcal{V}\) be a non-empty open set in \((\pp{A}{d},\we)\). By Theorem \ref{Theorem 4.10}, there exists a convex combination
\[
\mu =\sum_{i=1}^{r}a_{i}\mu _{\xx_{i}},\quad a_{i}\ge 0,\quad  \sum_{i=1}^{r}a_{i}=1
\]
of measures \(\mu _{\xx_{i}}\in\dd_{A,(\nk)}\) which belongs to \(\mathcal{V}\). 
\par\smallskip
Let \(k_{i}\) be such that \((A^{N_{k_{i}}}-I)\xx_{i}=0\), \(1\le i\le r\). By Fact \ref{Fact 4.13}, \(\wh{B^{lm_{k_{i}}}\mu _{\xx_{i}}}(\hh_{0})=\muc_{\xx_{i}}(\hh_{0})\) for every \(l\ge 1\). Setting \(m_{0}=m_{k_{1}}\dots m_{k_{r}}\), we have that 
\(\) \(\wh{B^{lm_{0}}\mu _{\xx_{i}}}(\hh_{0})=\muc_{\xx_{i}}(\hh_{0})\) for every \(l\ge 1\) and every \(i\in\{1,\ldots,r\}\). Hence \(\wh{B^{lm_{0}}}(\hh_{0})=\muc(\hh_{0})\) for every \(l\ge 1\). 
\par\medskip
If \(\muc(\hh_{0})\neq 0\), 
it follows that \(\mu \) belongs to \(\gro{A,B}^{\,\gamma }\). 
If \(\muc(\hh_{0})=0\), the measure \(\mu _{\rho }:=(1-\rho )\mu +\rho\delta _{\pmb{0}}\) belongs to \(\mathcal{V}\) if \(0<\rho <1\) is sufficiently small, and \(\muc_{\rho }(\hh_{0})=\rho\neq 0\). Also \(\wh{B^{lm_{0}}\mu _{\rho }}(\hh_{0})=(1-\rho )\wh{B^{lm_{0}}\mu }(\hh_{0})+\rho =\wh{\mu _{\rho }}(\hh_{0})\) for every \(l\ge 0\), and hence \(\mu _{\rho }\) belongs to \(\gro{A,B}^{\,\gamma }\). We have thus shown that \(\gro{A,B}^{\,\gamma }\) is a dense \(G_{\delta }\) subset of \((\pp{A}{d},\we)\). 
\par\medskip
Any measure \(\mu \in\gro{A,B}^{\,\gamma }\) is such that
$\limsup_{n\rightarrow +\infty} |\wh{B^{n}\mu }(\hh_{0})|>0$, and hence (since \(\hh_{0}\neq \pmb{0}\)) such that 
\(B^{n}\mu \nflw \leb_{d}\) for every \(\mu \in\gro{A,B}^{\,\gamma }\). So the set 
\[
\gro{A,B}^{\,0 }=\bigl\{\mu \in\pp{A}{d}\;;\;B^{n}\mu \nflw\leb_{d}\bigr\}
\]
is residual in \((\pp{A}{d},\we)\).
\par\smallskip
In order to complete the proof, it remains to show the following analogue of Fact \ref{Fact 3.4}:
\begin{fact}\label{Fact 3.14}
 The set \(\pp{A,c}{d}\) is a dense \(G_{\delta }\) subset of \((\pp{A}{d},\we)\).
\end{fact}
\begin{proof}[Proof of Fact \ref{Fact 3.14}]
 As mentioned already in the proof of Fact \ref{Fact 3.4}, the set \(\pp{A,c}{d}\) is known to be a \(G_{\delta }\) subset of \((\pp{A}{d},\we)\), so that only its density remains to be proved. The argument for this follows closely the proof of \cite{S1}*{Th. 2}, and reproves at the same time that \(\pp{A,c}{d}\) is \(G_{\delta }\).
 \par\medskip 
 For every \(\tau >0\), let
\[
F_{\tau}:=\bigl\{\mu \in\pp{A}{d}\;;\;\exists\,\xx\in\td\ \textrm{such that}\ \mu (\{\xx\})\ge\tau \bigr\}.
\]
Then the set \(F_{\tau}\) is easily seen to be closed in \((\pp{A}{d},\we)\). Let us now show that \(F_{\tau}\) is nowhere dense. Let \((\nk)_{k\ge 1}\) be a strictly increasing sequence of \emph{prime numbers} such that \(N_{1}>1/\tau \). By Theorem \ref{Theorem 4.10}, the convex hull of the set 
\[
\dd_{A,(\nk)}=\bigl\{\mu _{\xx}\;;\;A^{\nk}\xx=\xx\ \textrm{for some}\ k\ge 1\bigr\}
\]
is dense in \((\pp{A}{d},\we)\). Hence, given a non-empty open subset \(\mathcal{V}\) of \((\pp{A}{d},\we)\), there exist vectors \(\xx_{1},\dots,\xx_{r}\) in \(\td\), integers $k_1,\dots, k_r\ge 1$ and coefficients \(a_{1},\dots,a_{r}\ge 0\) with \(\sum_{i=1}^{r}a_{i}=1\) and \(A^{N_{k_{i}}}\xx_{i}=\xx_{i}\) for each \(1\le i\le r\) such that 
\[
\mu =\sum_{i=1}^{r}a_{i}\mu _{\xx_{i}}\quad \textrm{belongs to}\ \mathcal{V}.
\]
Since \(N_{k_{i}}\) is prime, the minimal period of \(\xx_{i}\) is \(N_{k_{i}}\), and thus 
\[
\mu _{\xx_{i}}(\{\xx\})\le\dfrac{1}{N_{k_{i}}}<{\tau }\quad \textrm{for every}\ \xx\in\td.
\]
It follows that \(\mu (\{\xx\})<\tau \) for every \(\xx\in\td\), and \(\mu \) does not belong to \(F_{\tau}\). So \(F_{\tau}\) is nowhere dense in \((\pp{A}{d},\we)\), and 
\[
\pp{A,c}{d}=\pp{A}{d}\setminus\bigcup_{l\ge 1}F_{2^{-l}}
\]
is a dense \(G_{\delta }\) subset of \((\pp{A}{d},\we)\) by the Baire Category Theorem.
\end{proof}

The proof of Theorem \ref{Theorem 3} is completed by combining Fact \ref{Fact 3.14} with the assertion that \(\gro{A,B}^{0}\) is residual in \((\pp{A}{d},\we)\).\hfill\(\square\)

\begin{remark}\label{Remark 3.15}
 The proof of Fact \ref{Fact 3.14} would apply equally well to Fact \ref{Fact 3.4}, but since the result is more standard in the one-dimensional case, we preferred to mention the classical arguments.
\end{remark}

\section{Further results and remarks}\label{Section 5}
\subsection{A complement to a result of Johnson and Rudolph}\label{Section 5.1}
Let \(p\ge 2\) be an integer, and let \(\cn_{n\ge 0}\) be a sequence of positive integers. We have recalled in the introduction and in Section \ref{Section 3} conditions on \(\cn\) implying that each measure \(\mu \in\pp{p}{}\) which is ergodic and of positive entropy is \(\cn\)-generic --- thus showing that the set
\[
G'_{p,\cn}:=\bigl\{\mu \in\pp{p}{}\;;\;T_{c_{n}}\mu \flw\leb\quad  \textrm{along a sequence of upper density}\ 1\bigr\}
\]
is residual in \((\pp{p}{},\we)\). 
We present here an alternative harmonic analysis approach to this kind of result. It has the benefit of circumventing the arguments that depend on positive entropy, when applicable.

\begin{theorem}\label{Theorem 5.10}
 Let \(p\ge 2\), and let \({\cn}_{n\ge 0}\) be a sequence of integers  satisfying the following condition:
\par\medskip
 \((\star)\)\hfill
 \begin{minipage}[c]{13cm}
  there exists a sequence \((\mu_{k})_{k\ge 1}\) of elements of \(\pp{p}{}\) such that 
  \(\mu _{k}\flw \delta _{1}\), and moreover, the set 
  \(
\bigl\{n\ge 1\;;\;|\muc_{k}(a.c_{n})|<\varepsilon \bigr\}
\)
has density \(1\) for every \(a\in\Z\setminus\{0\}\), every \(\varepsilon >0\) and every \(k\ge 1\).
 \end{minipage}
 \par\medskip
\noindent Then the set \(G'_{p,\cn}\) is residual in \((\pp{p}{},\we)\).
\end{theorem}

A word about terminology: saying that a sequence \((\nu _{n})_{n\ge 1}\) of measures converges to \(\nu \) along a subset of upper density \(1\) means that for any neighborhood \(\mathcal{V}\) of \(\nu \) in \(\mathcal{P}(\T)\), the set \(\{n\ge 1\;;\;\nu _{n}\in \mathcal{V}\}\) has upper density \(1\), i.~e. 
\[
\limsup_{N\to+\infty}\,\dfrac{1}{N}\,\#\,\bigl \{ 1\le n\le N\;;\;\nu _{n}\in \mathcal{V}\bigr\}=1. 
\]
This is equivalent to the following property: for any \(a_{0}\ge 1\) and any \(\varepsilon >0\), the set 
\[
\bigl \{ n\ge 1\;;\;|\wh{\nu }_{n}(a)-\wh{\nu }(a)|<\varepsilon \quad \textrm{for every}\ a\in\Z\ \textrm{with}\ |a|\le a_{0}\bigr\} 
\]
has upper density \(1\). 
In this case, one can construct a strictly increasing sequence \((N_{k})_{k\ge 1}\) of integers such that
\[
\dfrac{1}{N_{k}}\,\#\,\bigl \{ 1\le n\le N_{k}\;;\;\forall\,|a|\le k, \ |\wh{\nu }_{n}(a)-\wh{\nu }(a)|<2^{-k} \bigr\}\ge 1-2^{-k}
\]
for every \(k\ge 1\), and \(N_{k+1}\ge 2^{k}N_{k}\). If we consider the strictly increasing sequence \((n_{j})_{j\ge 1}\) obtained by enumerating the set 
\[
D=\bigcup_{k\ge 1}\bigl \{ N_{k-1}<n\le N_{k}\;;\;\forall\,|a|\le k, \ |\wh{\nu }_{n}(a)-\wh{\nu }(a)|<2^{-k}\bigr\} 
\]
(with the convention that \(N_{0}=0\)), we obtain that \(D=\{n_{j}\;;\;j\ge 1\}\) has upper density \(1\) and that 
\(\wh{\nu }_{n_{j}}(a)\longrightarrow\wh{\nu }(a)\) as \(j\longrightarrow+\infty\) for every \(a\in\Z\).

\begin{proof}[Proof of Theorem \ref{Theorem 5.10}]
We first observe that the set \(G'_{p,\cn}\) can be written as \(G'_{p,\cn}=\widetilde{G}_{p,\cn}\cap\mathcal{P}_{p,c}(\T)\), where
\begin{align*}
\widetilde{G}_{p,\cn}=\bigl \{ &\mu \in\mathcal{P}_{p}(\T)\;;\;\forall\,N_{0},a_{0}\ge 1,\quad \forall\,\varepsilon ,\delta \in (0,1)\cap\Q\\
&\exists\,N>N_{0},\quad \exists\,F\subseteq\{1,\dots,N\}\quad \textrm{with}\ \#\,F\ge(1-\delta )N\\
&\textrm{such that}\ \forall\,a\in\Z\ \textrm{with}\ 0<|a|\le a_{0},\ \forall\,n\in F, \;|\wh{\mu }(a.c_{n})|<\varepsilon \bigr\} .
\end{align*}
The set \(\widetilde{G}_{p,\cn}\) is clearly \(G_{\delta} \) in \((\mathcal{P}_{p}(\T),w^{*})\). Since \(\mathcal{P}_{p,c}(\T)\) is residual in \(\mathcal{P}_{p}(\T)\), it suffices to show that \(\widetilde{G}_{p,\cn}\) is dense in \(\mathcal{P}_{p}(\T)\). In order to do this, we are going to exhibit a dense set of measures \(\mu \in\mathcal{P}_{p}(\T)\) with the following property:
\begin{equation}\label{Equation 9}
 \forall\,a\in\Z\setminus\{0\},\ \forall\,\varepsilon >0,\  \textrm{the set}\ \{n\ge 1\;;\;|\wh{\mu }(ac_{n})|<\varepsilon \}\ \textrm{has density 1.}
\end{equation}
Since the intersection of finitely many sets of density \(1\) is again of density \(1\), the measures in this set will be such that
\begin{equation}\label{Equation 10}
 \forall\,a_0\geq 1,\ \forall\,\varepsilon >0,\ \textrm{the set}\ \{n\ge 1\;;\;\forall\,0<|a|\le a_{0},\ |\wh{\mu }(ac_{n})|<\varepsilon \}\ \textrm{has density 1}
\end{equation}
and hence upper density \(1\). Such measures will hence belong to the set \(G'_{p,\cn}\). 

Our assumption \((\star)\) states that the measures \(\mu _{k}\), \(k\ge 1\), satisfy (\ref{Equation 9}). Fix \(\nu \in\pp{p}{}\), and set \(\nu _{k}=\mu _{k}*\nu \) for every \(n\ge 1\). For any 
\(\varepsilon >0\), the set 
\(\{n\ge 1\;;\;|\wh{\nu }_{k}(a.c_{n})|<\varepsilon \}\) has density $1$, and it follows that the measures \(\nu _{k}\) satisfy (\ref{Equation 9}). Since \(\nu _{k}\flw\nu \) as \(\fl{k}{+\infty,}\) this concludes the proof of Theorem \ref{Theorem 5.10}.
\end{proof}

Theorem \ref{Theorem 5.10} applies for instance to the case where \(c_{n}=q^{n}\), \(n\ge 0\), provided that \(p,q\ge 2\) are two multiplicatively independent integers, and allows to retrieve \cite{JR}*{Th. 8.2}, which states that $G'_{p,(q^n)}$ is residual in \((\mathcal{P}_{p}(\T),w^{*})\). 

To this aim, it suffices to exhibit a sequence \((\mu _{k})_{k\ge 1}\) of measures from \(\pp{p}{}\) satisfying (\ref{Equation 9}) and such that \(\mu _{k}\flw\delta _{1}\).
The measures that we shall consider are the Bernoulli convolutions \(\mu _{\Theta }\) introduced at the end of the proof of Theorem \ref{Theorem 1}, where \(\Theta =(\theta _{0},\dots,\theta _{p-1})\) is a \(p\)-tuple of elements of \((0,1)\) with \(\sum_{j=0}^{p-1}\theta _{j}=1\). 
They are \(T_{p}\)-invariant, and
\[
\wh{\mu}_\Theta (m)=\prod_{n\ge 1}\Bigl ( \theta _{0}+\sum_{j=1}^{p-1}\theta _{j}e^{2i\pi mjp^{-n}}\Bigr) \quad\textrm{for every } m\in\Z.
\]
It is shown by Lyons in \cite{L1} and by Feldman and Smorodinsky in \cite{FS} that \(T_{q^{n}}\mu _{\Theta }\flw\leb\)
as \(\fl{n}{+\infty.}\) Since \(\mu _{\Theta }\flw\delta _{1}\) as 
\(\fl{\Theta }{(1,0,\dots,0),}\) assumption \((\star)\) from Theorem \ref{Theorem 5.10} is satisfied, and \(G'_{p,(q^{n})}\) is residual in \((\pp{p}{},\we)\).
\par\medskip
The behaviour of the Fourier coefficients of Bernoulli convolutions has been studied extensively, particularly when $p$ is equal to $2$ or $3$;
see for instance the classical book \cite{KS} by Kahane and Salem. The reader is also encouraged to have a look at related recent papers like \cites{Bremont, VarjuYu} and the references therein. An important work on the subject is that of Blum and Epstein \cite{BE}, where the authors provide upper and lower bounds on 
\(|\wh{\mu}_{\Theta }(m)|^{2}\) which allow them to give a characterisation of sequences of positive integers along which \(\wh{\mu}_{\Theta  }(m)\) tends to \(0\). In the case \(p=2\),  this characterisation is given in terms of the order of magnitude of \(R(m)\), which is the number of \emph{runs}, \emph{i.~e.} of maximal blocks of the same digit \(0\) or \(1\), appearing in the binary expansion of \(m\). Equivalently, \(R(m)\) is the number of digits changes in the binary expansion of \(m\). Then as soon as \(\Theta \neq(1/2,1/2)\) (in which case \(\mu _{\Theta }\) is the Lebesgue measure on \(\T\)), there ex ist two constants \(C_{1},C_{2}>0\) such that, for every \(m\in\Z\),
\[
\exp(-C_{2}R(m))\le|\wh{\mu }_{\Theta }(m)|\le\exp(-C_{1}R(m)).
\]
It follows that if \((m_{k})\) is any strictly increasing sequence of integers, we have \(\wh{\mu }_{\Theta }(m_{k})\longrightarrow 0\) as \(k\longrightarrow +\infty\) if and only if \(R(m_{k})\longrightarrow+\infty\) as \(k\longrightarrow+\infty\).
\par\medskip
For general \(p\), here is the result proved by Blum and Epstein in \cite{BE}:

\begin{theorem}[\cite{BE}]\label{Theorem 5.11}
Let $p\ge 2$.
 Let \(\Theta =(\theta _{0},\dots,\theta _{p-1})\) be a \(p\)-tuple of elements from \((0,1)\) with the property that the polynomial \(Q_{\Theta }(z)=\sum_{j=0}^{p-1}\theta _{j}z^{j}\) does not vanish on \(\T\). Let, for \(m\ge 0\),
 \[
\psi (m)=R_{0}(m)+R_{p-1}(m)+N(m),
\]
where \(R_{0}(m)\) is the number of maximal blocks of $0$s, $R_{p-1}(m)$ is the number of maximal blocks of $(p-1)$s, and $N(m)$ is the number of digits other than $0$ and $p-1$ in the expansion of $m$ in base $p$. Then there exist two constants $C_{1}, C_{2}>0$ such that
\[
\exp(-C_{2}\psi (m))\le|\wh{\mu }_{\Theta }(m)|\le\exp(-C_{1}\psi (m))\quad \textrm{for every}\ m\in\Z.
\]
Hence $\wh{\mu }_{\Theta }(m_{k})\longrightarrow 0$ as $k\longrightarrow +\infty$ if and only if $\psi (m_{k})\longrightarrow+\infty$ as $k\longrightarrow +\infty$.
\end{theorem}

As a consequence, we obtain the following result, which we state using the notation from Theorem \ref{Theorem 5.11}:

\begin{proposition}\label{Proposition 5.12}
 Let \(\cn_{n\ge 0}\) be a sequence of integers such that, for every \(a\in\Z\setminus\{0\}\), the sequence $(\psi (a.c_{n}))_{n\ge 0}$ tends to infinity along a sequence of density \(1\). Then the set \(G'_{p,\cn}\) is residual in \((\pp{p}{},\we)\).
\end{proposition}

\begin{proof}
 It suffices to show that \(\cn_{n\ge 0}\) satisfies assumption \((\star)\) of Theorem \ref{Theorem 5.10}. Let \(\Theta =(\theta _{0},\dots,\theta _{p-1})\) be a sequence of elements of \((0,1)\) summing up to \(1\), and suppose that \(\theta _{0}>1/2\). Then 
 \(\sum_{j=1}^{p-1}\theta _{j}<1/2\), and hence we have
 \[
|Q_{\Theta }(z)|\ge \theta _{0}-\sum_{j=1}^{p-1}\theta _{j}|z|^{j}>0
\]
for every \(z\in\C\) with \(|z|\le 1\). Thus the polynomial \(Q_{\Theta }\) does not vanish on \(\T\). Let \((\Theta _{k})_{k\ge 1}\) be such that \(\fl{\Theta _{k}}{(1,0,\dots,0)}\) as \(\fl{k}{+\infty, }\) and \(Q_{\Theta _{k}}\) does not vanish on \(\T\). Then \(\mu _{\Theta _{k}}\flw\delta _{1}\) as \(\fl{k}{+\infty}\). Moreover, the assumption of Proposition \ref{Proposition 5.12} combined with Theorem \ref{Theorem 5.11} implies that for every \(a\in\Z\setminus\{0\}\), \(\fl{\muc_{\Theta _{k}}(ac_n)}{0}\) along a sequence of density \(1\). The assumption \((\star)\) is thus satisfied, and Proposition \ref{Proposition 5.12} follows.
\end{proof}

\begin{remark}
If \(p,q\ge 2\) are multiplicatively independent integers, and \(c_{n}=q^{n}\), \(n\ge 0\), it is shown in \cite{FS}*{Prop. 1} that any word \(w\) in the letters \(0,1,\dots,p-1\) appears in the expansion of \(q^{n}\) in base $p$ for every  $n$ belonging to a set of integers of density \(1\). The proof given there can be extended to show that for any \(a\in\Z\setminus\{0\}\), \(w\) appears in the 
expansion of \(a.q^{n}\) in base $p$ for every $n$ belonging to a set of integers of density \(1\).
\end{remark}

\subsection{Some open questions}
In Sections \ref{Section 2} and \ref{Section 3}, numerous examples of sequences $\cn$ were presented, for which the set
\[
\gro{p,\cn}{}=\bigl\{\mu \in\pp{p,c}{}\;;\;T_{c_{n}}\mu \nflw\leb\bigr\} 
\]  
was found to be a residual subset of \((\pp{p}{},\we)\).
The condition (H) presented in Section \ref{Section 2} is the most general one that we can provide for the residuality of $\gro{p,\cn}{}$ to hold. However, it does not apply to \emph{all} sequences \(\cn_{n\ge 0}\), leaving the following intriguing question unanswered.

\begin{question}\label{Question 1}
 Let \(p\ge 2\), and let \(\cn_{n\ge 0}\) be any strictly increasing sequence of integers. Is it true that the set \(\gro{p,\cn}{}\) is residual in \((\pp{p}{},\we)\)? 
\end{question}

\par\medskip
If \(\mu \) is \(\cn\)-generic, then 
\(
\frac{1}{N}\sum_{n=0}^{N-1}T_{c_{n}}\mu \flw\leb\quad \textrm{as}\ \fl{N}{+\infty}.
\)
It is thus natural to consider the set 
\[
\gro{p,\cn}''=\bigl\{\mu \in\pp{p,c}{}\;;\;\fr{}\sum_{n=0}^{N-1}T_{c_{n}}\mu \flw\leb\quad \textrm{as } N\longrightarrow +\infty\bigr\}
\]
and to ask the following question.
\begin{question}\label{Question 2}
 Let \(p\ge 2\), and let \(\cn_{n\ge 0}\) be a strictly increasing sequence of integers. Is the set \(\gro{p,\cn}''\) residual in \((\pp{p}{},\we)\)?
\end{question}
For the sequences considered in \cites{H,JR,M,LMP}, the density of the set \(\gro{p,\cn}''{}\) in \((\pp{p}{},\we)\) follows from the result that ergodic measures of positive entropy in \(\pp{p}{}\) are \(\cn\)-generic. But the question of the residuality remains widely open, and it is actually not known if the set \(\gro{p,(q^{n})}''{}\) is residual in \((\pp{p}{},\we)\) when \(q\ge 2\) is an integer which is multiplicatively independent from \(p\). This question is also connected to another conjecture from \cite{L2}, called (C7), which runs as follows and seems to be still open: 
\begin{conjecture}
\emph{
 Let \(p,q\ge 2\) be multiplicatively independent  integers. For any measure \(\mu \in \pp{p,c}{}\), \(\mu \)-almost every \(x\in\T\) is normal in base \(q\), i.e.
 \[
\frac{1}{N}\sum_{n=0}^{N-1}e^{2i\pi aq^{n}x} \to 0 \quad  \textrm{as} \quad N \to \infty \quad  \textrm{for every}\ a\in\Z\setminus\{0\}.
\]
}
\end{conjecture}

\par\medskip
In our examination of Conjecture~\ref{conj:C3} in the multidimensional context, we have established conditions on matrices \(A,B\in M_{d}(\Z)\) with non-zero determinant that imply that the set
\[
\gro{A,B}{}=\bigl\{\mu \in\pp{A,c}{}\;;\;T_{B^{n}}\mu \nflw\leb_{d}\bigr\}
\]
is residual in \((\pp{A}{d},\we)\). These conditions (that \(A\) be diagonalisable in \(M_{d}(\C)\), that \(\det A\) and \(\det B\) be relatively prime...) arise due to technical difficulties in the proofs in the higher dimensional case. However, it may be that these conditions are not necessary; this is true in the one-dimensional setting.

\begin{question}\label{Question 3}
 Let \(d\ge 2\) and \(A,B\in M_{d}(\Z)\) with $\det A\neq 0$ and $\det B\neq 0$. Is it true that the set \(\gro{A,B}\) is residual in \((\pp{A}{d},\we)\)? 
\end{question}
Let also
\[
\gro{A,B}''=\bigl\{\mu \in\pp{A,c}{d}\;;\;\fr{}\sum_{n=0}^{N-1}T_{B^{n}}\mu \flw\leb_{d}\quad \textrm{as } N\longrightarrow +\infty\bigr\}.
\]
\par\medskip
In analogy to Question \ref{Question 2} , one may also ask:
\begin{question}\label{Question 4}
 Let \(d\ge 2\) and \(A,B\in M_{d}(\Z)\) with $\det A\neq 0$ and $\det B\neq 0$. Is it true that the set \(\gro{A,B}''\) is residual in \((\pp{A}{d},\we)\)?
\end{question}




\begin{dajauthors}
\begin{authorinfo}[C.B.]
  Catalin Badea \& Sophie Grivaux\\
  Univ. Lille, CNRS\\
  UMR 8524 - Laboratoire Paul
Painlev\'e\\ Lille, France\\
  cbadea\imageat{}univ-lille\imagedot{}fr \\
  sophie.grivaux\imageat{}univ-lille\imagedot{}fr \\
  \url{https://pro.univ-lille.fr/catalin-badea}\\
  \url{https://pro.univ-lille.fr/sophie-grivaux}
\end{authorinfo}
\end{dajauthors}


\begin{bibdiv}
  \begin{biblist}
  
\bib{Alg}{article}{
   author={Algom, Amir},
   title={A simultaneous version of Host's equidistribution theorem},
   journal={Trans. Amer. Math. Soc.},
   volume={373},
   date={2020},
   number={12},
   pages={8439--8462},
}

\bib{Algom}{article}{
   author={Algom, Amir},
   title={Actions of diagonal endomorphisms on conformally invariant
   measures on the 2-torus},
   journal={Monatsh. Math.},
   volume={195},
   date={2021},
   number={4},
   pages={545--564},
} 

\bib{AlonPeres}{article}{
   author={Alon, N.},
   author={Peres, Y.},
   title={Uniform dilations},
   journal={Geom. Funct. Anal.},
   volume={2},
   date={1992},
   number={1},
   pages={1--28},
}

\bib{BadGri}{article}{
   author={Badea, Catalin},
   author={Grivaux, Sophie},
   title={Kazhdan constants, continuous probability measures with large
   Fourier coefficients and rigidity sequences},
   journal={Comment. Math. Helv.},
   volume={95},
   date={2020},
   number={1},
   pages={99--127},
}
  
  \bib{BenQui}{article}{
   author={Benoist, Yves},
   author={Quint, Jean-Fran\c{c}ois},
   title={Introduction to random walks on homogeneous spaces},
   journal={Jpn. J. Math.},
   volume={7},
   date={2012},
   number={2},
   pages={135--166},
}
  
  \bib{Ber}{article}{
   author={Berend, Daniel},
   title={Multi-invariant sets on tori},
   journal={Trans. Amer. Math. Soc.},
   volume={280},
   date={1983},
   number={2},
   pages={509--532},
}

\bib{BerPerJLMS}{article}{
   author={Berend, Daniel},
   author={Peres, Yuval},
   title={Asymptotically dense dilations of sets on the circle},
   journal={J. London Math. Soc. (2)},
   volume={47},
   date={1993},
   number={1},
   pages={1--17},
}

\bib{BE}{article}{
   author={Blum, J. R.},
   author={Epstein, Bernard},
   title={On the Fourier-Stieltjes coefficients of Cantor-type
   distributions},
   journal={Israel J. Math.},
   volume={17},
   date={1974},
   pages={35--45},
}


\bib{Bosh}{article}{
   author={Boshernitzan, Michael D.},
   title={Elementary proof of Furstenberg's Diophantine result},
   journal={Proc. Amer. Math. Soc.},
   volume={122},
   date={1994},
   number={1},
   pages={67--70},
}

\bib{BLMV}{article}{
   author={Bourgain, Jean},
   author={Lindenstrauss, Elon},
   author={Michel, Philippe},
   author={Venkatesh, Akshay},
   title={Some effective results for $\times a\times b$},
   journal={Ergodic Theory Dynam. Systems},
   volume={29},
   date={2009},
   number={6},
   pages={1705--1722},
}

\bib{B}{article}{
   author={Bowen, Rufus},
   title={Periodic points and measures for Axiom $A$ diffeomorphisms},
   journal={Trans. Amer. Math. Soc.},
   volume={154},
   date={1971},
   pages={377--397},
}

\bib{Bremont}{article}{
   author={Br\'{e}mont, Julien},
   title={Self-similar measures and the Rajchman property},
   journal={Ann. H. Lebesgue},
   volume={4},
   date={2021},
   pages={973--1004},
} 
		
\bib{DGS}{book}{
   author={Denker, Manfred},
   author={Grillenberger, Christian},
   author={Sigmund, Karl},
   title={Ergodic theory on compact spaces},
   series={Lecture Notes in Mathematics, Vol. 527},
   publisher={Springer-Verlag, Berlin-New York},
   date={1976},
   pages={iv+360},
}
\bib{EinsFish}{article}{
   author={Einsiedler, Manfred},
   author={Fish, Alexander},
   title={Rigidity of measures invariant under the action of a
   multiplicative semigroup of polynomial growth on $\mathbb{T}$},
   journal={Ergodic Theory Dynam. Systems},
   volume={30},
   date={2010},
   number={1},
   pages={151--157},
}
	

\bib{EnsLind}{article}{
   author={Einsiedler, M.},
   author={Lindenstrauss, E.},
   title={Diagonal actions on locally homogeneous spaces},
   conference={
      title={Homogeneous flows, moduli spaces and arithmetic},
   },
   book={
      series={Clay Math. Proc.},
      volume={10},
      publisher={Amer. Math. Soc., Providence, RI},
   },
   date={2010},
   pages={155--241},
}

\bib{Feldman}{article}{
   author={Feldman, J.},
   title={A generalization of a result of R. Lyons about measures on
   $[0,1)$},
   journal={Israel J. Math.},
   volume={81},
   date={1993},
   number={3},
   pages={281--287},
}
	

\bib{FS}{article}{
   author={Feldman, J.},
   author={Smorodinsky, M.},
   title={Normal numbers from independent processes},
   journal={Ergodic Theory Dynam. Systems},
   volume={12},
   date={1992},
   number={4},
   pages={707--712},
}
	
  
  \bib{F}{article}{
   author={Furstenberg, Harry},
   title={Disjointness in ergodic theory, minimal sets, and a problem in
   Diophantine approximation},
   journal={Math. Systems Theory},
   volume={1},
   date={1967},
   pages={1--49},
}


%

\bib{Glasner}{article}{
   author={Glasner, Shmuel},
   title={Almost periodic sets and measures on the torus},
   journal={Israel J. Math.},
   volume={32},
   date={1979},
   number={2-3},
   pages={161--172},
}	

\bib{GM}{article}{
   author={Grivaux, Sophie},
   author={Matheron, \'{E}tienne},
   title={Invariant measures for frequently hypercyclic operators},
   journal={Adv. Math.},
   volume={265},
   date={2014},
   pages={371--427},
}
		

\bib{Hochman}{article}{
   author={Hochman, Michael},
   title={Geometric rigidity of $\times m$ invariant measures},
   journal={J. Eur. Math. Soc.},
   volume={14},
   date={2012},
   number={5},
   pages={1539--1563},
}

\bib{HochShm}{article}{
   author={Hochman, Michael},
   author={Shmerkin, Pablo},
   title={Equidistribution from fractal measures},
   journal={Invent. Math.},
   volume={202},
   date={2015},
   number={1},
   pages={427--479},
}

\bib{Hoch}{article}{
   author={Hochman, Michael},
   title={A short proof of Host's equidistribution theorem},
   journal={Israel J. Math.},
   volume={251},
   date={2022},
   number={2},
   pages={527--539},
}

  
  \bib{H}{article}{
   author={Host, Bernard},
   title={Nombres normaux, entropie, translations},
   journal={Israel J. Math.},
   volume={91},
   date={1995},
   number={1-3},
   pages={419--428},
}

\bib{H2}{article}{
   author={Host, Bernard},
   title={Some results of uniform distribution in the multidimensional
   torus},
   journal={Ergodic Theory Dynam. Systems},
   volume={20},
   date={2000},
   number={2},
   pages={439--452},
}

\bib{J}{article}{
   author={Johnson, Aimee S. A.},
   title={Measures on the circle invariant under multiplication by a
   nonlacunary subsemigroup of the integers},
   journal={Israel J. Math.},
   volume={77},
   date={1992},
   number={1-2},
   pages={211--240},
}
	
	\bib{JR}{article}{
   author={Johnson, Aimee},
   author={Rudolph, Daniel J.},
   title={Convergence under $\times_q$ of $\times_p$ invariant measures on
   the circle},
   journal={Adv. Math.},
   volume={115},
   date={1995},
   number={1},
   pages={117--140},
}



\bib{KS}{book}{
   author={Kahane, Jean-Pierre},
   author={Salem, Rapha\"{e}l},
   title={Ensembles parfaits et s\'{e}ries trigonom\'{e}triques},
   edition={2},
   note={With notes by Kahane, Thomas W. K\"{o}rner, Russell Lyons and Stephen
   William Drury},
   publisher={Hermann, Paris},
   date={1994},
   pages={245},
}

\bib{KS96}{article}{
   author={Katok, A.},
   author={Spatzier, R. J.},
   title={Invariant measures for higher-rank hyperbolic abelian actions},
   journal={Ergodic Theory Dynam. Systems},
   volume={16},
   date={1996},
   number={4},
   pages={751--778},
}

\bib{Kra}{article}{
   author={Kra, Bryna},
   title={A generalization of Furstenberg's Diophantine theorem},
   journal={Proc. Amer. Math. Soc.},
   volume={127},
   date={1999},
   number={7},
   pages={1951--1956},
}

\bib{KLO}{article}{
   author={Kwietniak, Dominik},
   author={\L \c{a}cka, Martha},
   author={Oprocha, Piotr},
   title={A panorama of specification-like properties and their
   consequences},
   conference={
      title={Dynamics and numbers},
   },
   book={
      series={Contemp. Math.},
      volume={669},
      publisher={Amer. Math. Soc., Providence, RI},
   },
   date={2016},
   pages={155--186},
}
		

\bib{Elon-padic}{article}{
   author={Lindenstrauss, Elon},
   title={$p$-adic foliation and equidistribution},
   journal={Israel J. Math.},
   volume={122},
   date={2001},
   pages={29--42},
}
		
	\bib{Lind}{article}{
   author={Lindenstrauss, Elon},
   title={Rigidity of multiparameter actions},
   journal={Israel J. Math.},
   volume={149},
   date={2005},
   pages={199--226},
}

\bib{LindMarg}{article}{
   author={Lindenstrauss, Elon},
   title={Recent progress on rigidity properties of higher rank
   diagonalizable actions and applications},
   conference={
      title={Dynamics, geometry, number theory---the impact of Margulis on
      modern mathematics},
   },
   book={
      publisher={Univ. Chicago Press, Chicago, IL},
   },
   date={2022},
   pages={362--425},
}

\bib{LMP}{article}{
   author={Lindenstrauss, Elon},
   author={Meiri, David},
   author={Peres, Yuval},
   title={Entropy of convolutions on the circle},
   journal={Ann. of Math. (2)},
   volume={149},
   date={1999},
   number={3},
   pages={871--904},
}

\bib{L1}{article}{
   author={Lyons, Russell},
   title={Mixing and asymptotic distribution modulo $1$},
   journal={Ergodic Theory Dynam. Systems},
   volume={8},
   date={1988},
   number={4},
   pages={597--619},
}


\bib{L2}{article}{
   author={Lyons, Russell},
   title={On measures simultaneously $2$- and $3$-invariant},
   journal={Israel J. Math.},
   volume={61},
   date={1988},
   number={2},
   pages={219--224},
}


		

\bib{M}{article}{
   author={Meiri, David},
   title={Entropy and uniform distribution of orbits in ${\bf T}^d$},
   journal={Israel J. Math.},
   volume={105},
   date={1998},
   pages={155--183},
}
	
	\bib{MP}{article}{
   author={Meiri, David},
   author={Peres, Yuval},
   title={Bi-invariant sets and measures have integer Hausdorff dimension},
   journal={Ergodic Theory Dynam. Systems},
   volume={19},
   date={1999},
   number={2},
   pages={523--534},
}
	
\bib{Parry}{article}{
   author={Parry, William},
   title={Squaring and cubing the circle---Rudolph's theorem},
   conference={
      title={Ergodic theory of ${\mathbb{Z}}^d$ actions},
      address={Warwick},
      date={1993--1994},
   },
   book={
      series={London Math. Soc. Lecture Note Ser.},
      volume={228},
      publisher={Cambridge Univ. Press, Cambridge},
   },
   date={1996},
   pages={177--183},
}
	
	
\bib{R}{article}{
   author={Rudolph, Daniel J.},
   title={$\times 2$ and $\times 3$ invariant measures and entropy},
   journal={Ergodic Theory Dynam. Systems},
   volume={10},
   date={1990},
   number={2},
   pages={395--406},
}


\bib{Shm}{arXiv}{
  author={Shmerkin, Pablo},
  title={Slices and distances: on two problems of Furstenberg and Falconer (to appear in the Proceedings of the 2022 ICM)},
  date={2022},
  eprint={2109.12157},
  archiveprefix={arXiv},
  primaryclass={math.CA},
}







\bib{S1}{article}{
   author={Sigmund, Karl},
   title={Generic properties of invariant measures for Axiom ${\rm A}$
   diffeomorphisms},
   journal={Invent. Math.},
   volume={11},
   date={1970},
   pages={99--109},
}

\bib{S2}{article}{
   author={Sigmund, Karl},
   title={On dynamical systems with the specification property},
   journal={Trans. Amer. Math. Soc.},
   volume={190},
   date={1974},
   pages={285--299},
}

\bib{VarjuYu}{article}{
   author={Varj\'{u}, P\'{e}ter P.},
   author={Yu, Han},
   title={Fourier decay of self-similar measures and self-similar sets of
   uniqueness},
   journal={Anal. PDE},
   volume={15},
   date={2022},
   number={3},
   pages={843--858},
} 

     \end{biblist}
\end{bibdiv}
\end{document}